\documentclass{amsart}    
\usepackage{amsmath,amsfonts,amssymb,amsopn,amscd}
\usepackage{bbm,wasysym,stmaryrd}
\usepackage{subfigure}
\usepackage{color,colortbl}
\usepackage{hyperref}
\usepackage[normalem]{ulem}
\usepackage{xcolor}
\usepackage{graphicx}
\graphicspath{{./Fig/}}

\renewcommand{\P}{\mathbbm{P}}

\newcommand{\Exp}{\mathbbm{E}}
\newcommand{\E}{\mathcal{E}}
\DeclareMathOperator{\eqlaw}{\stackrel{\mathcal{L}}{=}}

\newcommand{\G}{\mathcal{G}}

\newcommand{\R}{\mathbbm{R}}
\newcommand{\N}{\mathbbm{N}}

\newcommand{\A}{\mathcal{A}}

\newcommand{\M}{\mathcal{M}}
\newcommand{\C}{\mathcal{C}}

\definecolor{orange}{rgb}{1.00,0.50,0.0}

\theoremstyle{plain}
\newtheorem{theorem}{Theorem}[section]
\newtheorem{proposition}[theorem]{Proposition}
\newtheorem{lemma}[theorem]{Lemma}

\theoremstyle{definition}

\theoremstyle{remark}

\begin{document}	
\title[Universality in randomly connected neural networks]{Universality in randomly connected neural networks: old and new on a conjecture of Geman and Hwang}
\date{\today}
\author[J.~Touboul]{Jonathan Touboul}
\address{Department of Mathematics, Brandeis University, Waltham, MA, USA}

\begin{abstract}
Networks with random, heterogeneous interactions are ubiquitous in physics, biology and social sciences. A central question is whether, and how, their collective dynamics depend on the choice of the distribution of the random couplings. Forty years ago, Geman and Hwang formulated a precise version of this \emph{universality} question for randomly connected neural networks. We come back to these conjectures, review what was proven since, provide a few new extensions and propose new conjectures in this paper. For $n$ neuron networks with couplings with a mean scaling as $1/n$ and light enough tails, two classes of limits are expected. For \emph{weakly heterogeneous} couplings, with variance $o(1/n)$, the network is predicted to converge to a McKean--Vlasov equation, and a broad literature has established this under general assumptions, including for interacting diffusions on random graphs and for non-exchangeable systems. For \emph{strongly heterogeneous} couplings, with variance of order $1/n$, the limit is a self-consistent, non-Markovian equation driven by a Gaussian process whose law depends on the solution. In this domain, after landmark results for Gaussian couplings, universality was only recently established for linear interactions, while the continuous-time nonlinear case relevant to neuroscience remained open. We review these results and the main techniques involved, and propose some extensions towards the proof of the conjecture. In the weakly heterogeneous regime, we prove quenched convergence on every time interval under a fourth-moment assumption only, and convergence in probability under an assumption slightly stronger than a second moment. In the strongly heterogeneous regime, we prove convergence for nonlinear interactions in continuous time, averaged over sub-Gaussian couplings and on every time interval, by a relative-entropy comparison with Gaussian couplings. We then extend the conjecture to heavy-tailed couplings, and propose a limit driven by a self-consistent $\alpha$-stable process, for which we prove existence for every $\alpha\in(0,2)$ and uniqueness for $\alpha\in[1,2)$, and provide extensive numerical evidence in all three cases. Altogether, forty years after its formulation, Geman and Hwang's program remains an active field of research, with several basic questions open.
\end{abstract}

\maketitle

\bigskip

\hrule

\tableofcontents

\medskip 

\hrule

\section{Introduction}

Interacting systems with random interactions are ubiquitous in complex systems describing the coordinated interaction of large ensembles of agents. They model systems ranging from neuronal networks~\cite{sompolinsky-crisanti-etal:88,ostojic:14} to spins of atoms in ferromagnets~\cite{mezard1987spin} or the flashing of fireflies~\cite{strogatzSync,strogatz2000kuramoto}, as well as ecological communities~\cite{may:72,allesina-tang:12,bunin:17}, social dynamics~\cite{castellano-fortunato-loreto:09}, and large artificial neural networks, at initialization~\cite{hanin:23} and during training~\cite{mei-montanari-nguyen:18,sirignano-spiliopoulos:20}. In such systems, the dynamics of each individual element as well as the interaction with the others are generally heterogeneous. These heterogeneities may characterize differences in the element itself, or in the topology of the network; as such, they are quenched during the dynamics of the system, and are generally referred to as \emph{disorder}, in contrast with noise that can induce random \emph{dynamical} fluctuations. It is well-known that heterogeneities may have a dramatic effect on the collective dynamics. For instance, a classical result due to Kuramoto~\cite{kuramoto-arakai} showed that a fully-coupled network of oscillators may synchronize when the disorder in intrinsic frequencies is low, and lose synchrony above a certain level of disorder. In the context of neural networks, Sompolinsky and co-authors showed that the level of disorder in the coupling strength between identical neurons has a dramatic impact on the collective solution, and controls a transition between trivial dynamics (where all neurons converge to rest) and chaos~\cite{sompolinsky-crisanti-etal:88}. This transition gave rise to ample literature in neuroscience and mathematics, and explored various directions. Some of these focused on the possible contribution of the distribution of eigenvalues of the random coupling matrix of the dynamics, implicitly suggesting that the universality properties of the spectrum of random matrices~\cite{ginibre:65,bai:97,tao-vu:10}, could reflect a universality in the behavior of neural networks. Among works related to these ideas, authors considered the spectrum of connectivity matrices where the statistics of the random couplings reflect a population structure~\cite{rajan-abbott:06,aljadeff-stern-sharpee:15,hermann-touboul:12,cabana-touboul:13}, when the dynamics is bistable~\cite{stern-sompolinsky-abbott:14} or described by spiking models~\cite{kadmon-sompolinsky:15}. It is widely admitted in the physics literature that these phenomena may rely only on a few statistical properties of the disorder distribution, and that large universality classes of behaviors exist where the details of the distribution of the disorder do not matter in the large system size limit. This universality has found strong support for disorder at the individual dynamics level. For instance, in the Kuramoto model and with symmetric distributions of intrinsic frequencies, the phase transition only depends on the amplitude of the distribution at the average value (see e.g.~\cite{strogatz2000kuramoto} for a review). Regarding disorder in the coupling weights, much less is understood mathematically, but universality of behaviors is widely admitted, as expressed by M\'ezard, Parisi and Virasoro in~\cite[p.9]{mezard1987spin} in the context of the replica symmetric solution of the Sherrington-Kirkpatrick spin-glass model:
\begin{center}
\begin{minipage}{0.8\textwidth}
	``\emph{it is usually supposed that the distribution of the J's \emph{[coupling coefficients]} is Gaussian, although the same results can also be found for non-Gaussian distributions [...]: as far as the J's are quantities of order $1/n^{1/2}$, the results depend only on the first two moments of the distribution~\cite{sherrington-kirkpatrick:75}.}''
	\end{minipage}
	\end{center}
In the early 1980s in a remarkably ahead-of-time work, Geman~\cite{geman:82}, and Geman and Hwang~\cite{geman:1982b}, undertook to formulate and study mathematically this conjecture in the context of randomly connected neural networks. In this program, they noted that while dynamics can often be predicted using well established heuristic arguments consisting in considering that elements of a large-scale are `nearly independent' (Boltzmann's molecular chaos, or \emph{Sto\ss zahlansatz}), mathematically proving these conjectured limit raises substantial mathematical challenges. Quoting from~\cite{geman:1982b},
 \begin{center}
 \begin{minipage}{0.8\textwidth}
``\emph{Some of the results [are] surprisingly difficult to prove, but they are all easily anticipated from a ``chaos hypothesis'': i.e. an assumption of near independence for the components of the solutions of large systems of weakly coupled equations.}''
	 \end{minipage}
	 \end{center}

Geman and Hwang's statement has passed the test of time: over the forty years after its formulation, it has generated significant research, and while substantial progress has been made in the last decade (see Sections~\ref{sec:WeakLit}, \ref{sec:StrongLit} and~\ref{sec:HeavyLit}), the picture remains incomplete, and new conjectures extending Geman and Hwang's are still unexplored. The universality raises deep mathematical questions and has the potential to stimulate the development of new mathematical tools. 

The aim of the present paper is twofold. First, this paper surveys mathematical results related to this conjecture. We formulate them within an overarching formalism, and review some marking rigorous results obtained since. Although these questions extend to a variety of systems, our paper focuses on the dynamics of randomly connected neuronal networks with noise, for both historical (these systems are similar to the Geman-Hwang equations) and application reasons. The classical equation we shall study here is a firing-rate equation describing the dynamics of a network composed of $n$ neurons with random synaptic (more generally referred to as \emph{coupling}) weights~\cite{amari:71,sompolinsky-crisanti-etal:88} (see~\cite{cessac:10} for a review) through the activity of each neuron (a scalar variable $X^{i}_t$, its membrane potential, with firing rate $S(X^i_t)$):
\begin{equation}\label{eq:network}
  dX^{i,n}_{t} = \Big( f(X^{i,n}_{t}) +\sum_{j=1}^{n} J_{ij} S(X^{j,n}_{t})\Big)\,dt + \lambda dW^{i}_{t}
\end{equation}
The Lipschitz-continuous map\footnote{A variety of works sought to generalize results to $f$ non-Lipschitz, and seek conditions for the results to hold in these contexts. We do not consider these works and associated methods in the present manuscript.} $f$, generally taken to be equal to minus identity, describes the intrinsic dynamics of the {membrane potential}, generally an exponential return to a baseline voltage, the stochastic term accounts for the intrinsic fluctuations of the electrical cells activities (modeled as independent Brownian fluctuations with diffusion coefficient $\lambda$), and interactions occur through the firing rates {$S(X^j_t)$} of other cells, {where} the sigmoid $S$ is assumed to {be} a bounded Lipschitz continuous map. The interaction coefficients $J_{ij} \in \R$ are independent random variables\footnote{Other works have considered dropping the independence assumptions. These questions are not in the scope of this work.} whose probability distribution may depend on the network size. We discuss the possible limits of the above system as a function of (i) the tails of the probability distribution of the interaction coefficients and (ii) the scaling of the mean and possibly the variance (when finite) of the interaction coefficients as $n$ increases. {Second, we establish the following results towards the conjecture: (i)~in the weakly heterogeneous regime, quenched convergence for every scaling $n=o(a_n)$ under a fourth moment, and convergence in probability under slightly more than a second moment (Propositions~\ref{prop:OpNorm} and~\ref{prop:Truncation}); (ii)~in the strongly heterogeneous regime, averaged convergence for nonlinear interactions in continuous time, for independent sub-Gaussian couplings, on every time interval and with logarithmic rate (Theorem~\ref{thm:Convergence4}); (iii)~for heavy-tailed couplings, a formulation of the conjecture with a limit driven by a self-consistent $\alpha$-stable process, together with its existence for $\alpha\in(0,2)$ and its uniqueness for $\alpha\in[1,2)$ when $f$ is affine (Theorem~\ref{thm:HeavyWellPosed}); (iv)~numerical simulations supporting the conjecture in the three regimes, including for couplings that no theorem covers.} {Section~\ref{sec:Physics} recalls the heuristic proposed by Geman and Hwang, formulates the conjecture in three regimes (weakly heterogeneous, strongly heterogeneous and heavy-tailed couplings), and provides numerical simulations illustrating the three regimes conjectured. The three sections that follow, Sections~\ref{sec:Weak}, \ref{sec:StrongHet} and~\ref{sec:HeavyTails}, focus on each of the regime individually: we review the literature, state the new results of this paper and discuss the open problems. All proofs are given in the appendices.

\section{Sto\ss zahlansatz, universal limits and a generalized conjecture}\label{sec:Physics}

As highlighted by Geman and Hwang~\cite{geman:82,geman:1982b} and discussed in~\cite{cessac-samuelides:07b}, Boltzmann's \emph{Sto\ss zahlansatz}, or molecular chaos hypothesis, suggests that as $n\to\infty$, neurons become mutually independent and independent of the couplings, possibly because the influence of any given neuron on another one is overwhelmed by the mean-field interaction generated by the rest of the network. Mathematically, this corresponds to \emph{propagation of chaos}: for particle systems with independent initial conditions, any finite subset of particles becomes independent in the limit, a property that is equivalent, for exchangeable systems, to the convergence of the empirical measure~\cite[Proposition~2.2]{sznitman:89} (see~\cite{kac:56,mckean:66,jabin-wang:17,chaintron-diez:22} for origins and reviews). Loosely speaking, assuming that this property holds approximately, the interaction term $\sum_{j}J_{ij}S(X^{j,n}_t)$ is a sum of (approximately) independent and identically distributed variables, and its limit follows from a law of large numbers, a central limit theorem or a generalized central limit theorem, depending on the scaling and on the tails of the couplings.

For couplings with finite variance, we write $J_{ij}=\frac{\bar J}{n}+\frac{\sigma}{\sqrt{a_n}}\hat J_{ij}$, with $\hat J_{ij}$ centered with unit variance and $a_n\to\infty$; the scaling $1/n$ of the mean keeps the mean input finite, then
\[
\sum_{j=1}^n J_{ij} S(x^j_t) = \underbrace{\frac{\bar{J}}{n} \sum_{j=1}^n S(x^j_t)}_{\longrightarrow\ \bar{J}\,\Exp [ S(\bar{x}_t) ]} + \underbrace{\frac{\sigma}{\sqrt{a_n}} \sum_{j=1}^n \hat{J}_{ij} S(x^j_t)}_{\text{variance }\approx\ \sigma^2\frac{n}{a_n}\Exp[ S(\bar{x}_t)^2]}.
\]
The first term converges by the law of large numbers, and the fate of the second one depends on the ratio $n/a_n$. It vanishes when $n=o(a_n)$, the \emph{weakly heterogeneous} regime, in which the limit is the McKean--Vlasov equation of item~1 of the conjecture below. It converges, when $a_n=n$, to a centered Gaussian process with covariance $\sigma^2\Exp[S(\bar x_t)S(\bar x_s)]$, the \emph{strongly heterogeneous} regime, in which the limit is the self-consistent equation of item~2. For heavy-tailed couplings, the generalized central limit theorem replaces the Gaussian input by an $\alpha$-stable one.

At this heuristic level, the limits only depend on whether the sums of the couplings satisfy a law of large numbers or a (generalized) central limit theorem, and not on the precise distribution of the couplings: this is the universality conjecture. Mathematically, one needs to justify these limits under no extra assumption on the couplings, and to prove that the limit equations, which are non-Markovian, implicit equations in the space of stochastic processes, are well posed. In physics, these limits are derived with the Martin--Siggia--Rose generating functional or the dynamical cavity method~\cite{martin-siggia-rose:73,sompolinsky-zippelius:82,crisanti-sompolinsky:87,crisanti-sompolinsky:18,helias-dahmen:20}, the formalism that also underlies the dynamical theory of mean-field spin glasses~\cite{cugliandolo-kurchan:93}.

\subsection{Mathematical formulation of a conjecture}\label{sec:Conjecture}
Based on this ansatz, we can thus conjecture the following classes of universality for the limit system~\eqref{eq:network} as $n\to \infty$:

{\noindent\textbf{(A) Limiting distribution.} The limit law $Q$ depends on the regime as follows.}
\begin{enumerate}
\item {\bf Weakly heterogeneous interactions} When the row sums of the $J_{ij}$ satisfy a law of large numbers, i.e.\ $\sum_j J^n_{1j}$ converges in probability to a deterministic limit $\bar J$ (for instance, when the $J^n_{ij}$ have mean $\bar J/n$ and variance $o(1/n)$), the limit is given by a McKean--Vlasov type equation:
	\begin{equation}\label{eq:WeakMFequation}
	  dx_{t} = \Big( f(x_{t}) +\bar{J} \Exp[S(x_{t})]\Big)\,dt + \lambda dW_{t}. 
	\end{equation}
\item {\bf Strongly heterogeneous interactions} When the row sums of the $J_{ij}$ satisfy a central limit theorem, i.e.\ $\sum_j J^n_{1j}$ converges in law to a non-degenerate Gaussian variable $\mathcal N(\bar J,\sigma^2)$ (for instance, when the $J^n_{ij}$ have mean $\bar J/n$ and variance $\sigma^2/n$), the system converges towards the non-Markovian equation
	\begin{equation}\label{eq:StrongMFequation}
	  dx_{t} = \Big( f(x_{t}) +U^x_t\Big)\,dt + \lambda dW_{t}
	\end{equation}
	with $U^x_t$ a Gaussian process{, independent of $(x_0,W)$,} with mean $\bar J \mathbb{E}[S(x_t)]$ and covariance structure $K^x(t,s)=\sigma^{2}\mathbb{E}[S(x_{t})S(x_{s})]$ both depending on the solution to the equation. 
	\item {\bf Heavy-tailed interactions} If $J_{ij}=\frac{\mu}{n}+\frac{g_{ij}}{n^{1/\alpha}}$, where $\mu\in\R$ and the $g_{ij}$ are i.i.d.\ with a symmetric law ($g_{ij}$ and $-g_{ij}$ have the same distribution) and $\P(|g_{ij}|>x)\sim x^{-\alpha}$ for some $\alpha\in(0,2)$ (so that they belong to the domain of attraction of a symmetric $\alpha$-stable law), the system should converge towards the non-Markovian equation
\begin{equation*}
  dx_{t} = \Big( f(x_{t}) +\mu\,\Exp[S(x_t)]+Z^x_t\Big)\,dt + \lambda\, dW_{t},
\end{equation*}
where, denoting by $Q$ the law of the solution, the process $(Z^x_t)_{t\in[0,T]}$ is independent of $(x_0,W)$ and is a symmetric $\alpha$-stable process indexed by time, with finite-dimensional distributions
\begin{equation*}
\Exp\Big[e^{\,i\sum_{k=1}^K\theta_k Z^x_{t_k}}\Big]=\exp\Big(-c_\alpha \Exp\left[\Big|\sum_{k=1}^K\theta_k\, S(x_{t_k})\Big|^\alpha\right]\Big),
\end{equation*}
where $c_\alpha$ is the constant of the generalized central limit theorem for $g_{ij}$~\cite{gnedenko-kolmogorov:54} ($c_\alpha=\int_0^\infty u^{-\alpha}\sin u\,du=\Gamma(1-\alpha)\cos\frac{\pi\alpha}2$ for $\alpha\neq1$ and $c_1=\frac\pi2$
). 

This process has several other representations that we recall here. Equivalently, $Z^x_t=\int_{\C}S(y_t)\,M_\alpha(dy)$ for a symmetric $\alpha$-stable random measure $M_\alpha$ with control measure $c_\alpha Q$~\cite{samorodnitsky-taqqu:94}. It can also be defined as the LePage series~\cite{lepage-woodroofe-zinn:81}, which will be our preferred representation here for numerics and proofs:
\begin{equation*}
Z^x_t \eqlaw C_\alpha\sum_{k\geq 1}\varepsilon_k\,\Gamma_k^{-1/\alpha}\,S(y^k_t),
\end{equation*}
with $(\Gamma_k)_k$ the arrival times of a unit-rate Poisson process, $(\varepsilon_k)_k$ i.i.d. Rademacher signs and $(y^k)_k$ i.i.d. with law $Q$, all independent, and $C_\alpha$ a constant, equal to $1$ with our normalization.
\end{enumerate}
We note that the putative limit listed in the heavy-tailed conjecture is a complex stochastic process of a new type. We discuss existence and uniqueness of solutions in Section~\ref{sec:HeavyTails}. 

\medskip
\noindent{\textbf{(B) Mode of convergence.} Denote by $\hat\mu_n=\frac1n\sum_{i}\delta_{X^{i,n}}$ the empirical measure, by $Q^i_n(J)$ and $Q^{ij}_n(J)$ the laws of $X^{i,n}$ and of $(X^{i,n},X^{j,n})$ given the couplings $J$. In all three regimes, $\hat\mu_n\to Q$ in probability, and since the system is exchangeable, the averaged over the couplings law of $(X^{1,n},\dots,X^{k,n})$ converges to $Q^{\otimes k}$ for every fixed $k$~\cite[Proposition~2.2]{sznitman:89}. In terms of quenched results, we expect the convergences below in probability with respect to the couplings. Noting $d$ the bounded-Lipschitz distance\footnote{Averaged statements, for $\E_J[Q^i_n(J)]$ and $\E_J[Q^{ij}_n(J)]$, can also be derived from results below and convexity of $d$}, we conjecture, for fixed neurons $i\neq j$:}
\begin{enumerate}
	\item {In the two light-tailed regimes (items~1 and~2 of~(A)), the disorder is averaged out at the level of each neuron (self-averaging property), and propagation of chaos holds given the couplings: $Q^i_n(J)\to Q$ and $Q^{ij}_n(J)\to Q\otimes Q$.}
	\item {{In the heavy-tailed regime, we do not expect self-averaging: the input of each neuron is carried by finitely many large weights, so that $Q^i_n(J)$ should converge to a random law rather than to $Q$. Since two fixed neurons are joined by a large weight with vanishing probability, we expect $d\big(Q^{ij}_n(J),\,Q^i_n(J)\otimes Q^j_n(J)\big)\to0$ in probability, with random marginals.}}
\end{enumerate}

\subsection{Numerical exploration of the conjecture}\label{sec:Numerics}
We started by evaluating the conjecture numerically in a variety of situations, including those covered by available theorems in the weakly and strongly heterogeneous cases (Gaussian or light-tailed couplings) and situations that no theorem covers. In each case, we quantify the distance between finite networks and the predicted limit equation, and estimate how these compare to finite-size biases and sampling errors. More detail on numerical methods, algorithms and codes is given in Appendix~\ref{app:numerics}.

\begin{figure}[h]
	\includegraphics[width=.8\textwidth]{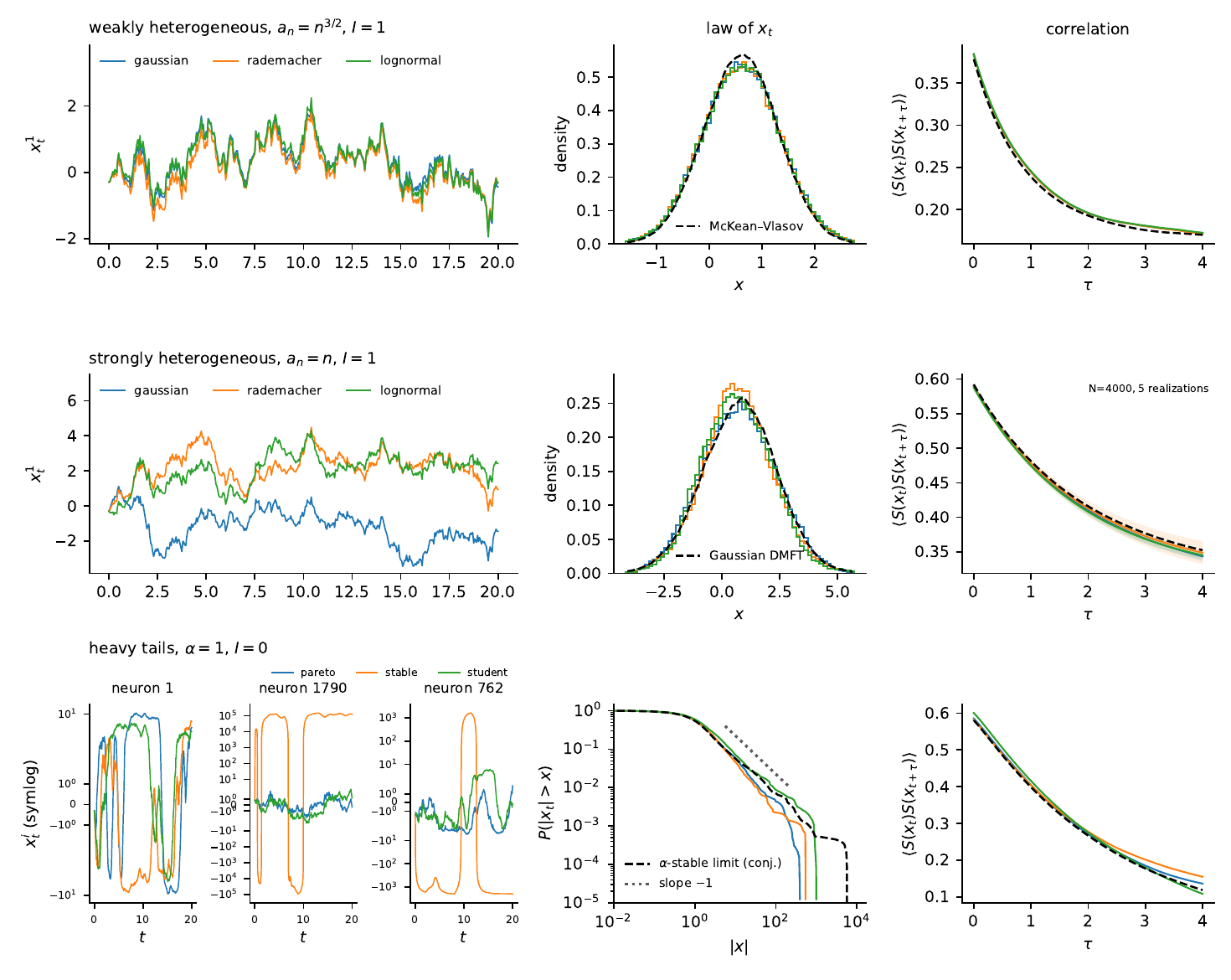}
		\caption{Trajectories and statistics of networks~\eqref{eq:network} in each of the three conjecture regimes. Colored curves: simulations with identical initial conditions and Brownian motions but distinct coupling distributions. Top: weakly heterogeneous couplings $J_{ij}=-1/n+2 g_{ij}/\sqrt{a_n}$ with $a_n=n^{3/2}$, $g_{ij}$ centered with unit variance (Gaussian, Rademacher, centered log-normal). Middle: same laws with $a_n=n$ (strongly heterogeneous regime). Bottom: heavy-tailed couplings $J_{ij}=g_{ij}n^{-1/\alpha}$, $\alpha=1$, with $g_{ij}$ symmetric Pareto, symmetric $\alpha$-stable or Student, normalized to the same tail $\mathbb{P}(|g|>x)\sim x^{-\alpha}$. Left: trajectory of neuron~1; bottom left: neuron~1 and the two neurons whose trajectories differ most between the three laws, in symmetric logarithmic scale. Center: law of $x^i_t$, pooled over neurons and $t\in[10,20]$ (bottom: tail distribution in log-log scale). Right: correlation function $\langle S(x^i_t)S(x^i_{t+\tau})\rangle$. Dashed black: numerical solution of the corresponding limit (McKean--Vlasov equation~\eqref{eq:WeakMFequation}, Gaussian self-consistent equation~\eqref{eq:StrongMFequation}, conjectured $\alpha$-stable equation of Section~\ref{sec:HeavyTails}). Parameters: $n=2\,000$ neurons with {$f(x)=-x+I$}, $S=\tanh$, $\lambda=1$, $I=1$ (top two rows) or $I=0$ (bottom row), and strong-regime correlation panel: $n=4\,000$, averaged over $5$ realizations of the couplings.}	
		\label{fig:universality}
\end{figure}

Figure~\ref{fig:universality} shows a few trajectories of the systems in the three regimes of universality, where, to highlight the impact of the coupling weights distribution, trajectories are simulated with the same initial conditions and Brownian motions. In the weakly heterogeneous regime, all trajectories converge to the same solution and the impact of the random part of the couplings vanishes. In the strongly heterogeneous regime, the trajectory of a given neuron depends on the realization of the weights, illustrating the impact of the random part of the couplings, but the law of the population and thus its correlation function do not depend on the choice of the coupling distribution. In particular, it coincides with the Gaussian self-consistent limit. With heavy tailed couplings, individual trajectories differ even more, the law of $x_t$ has power-law tails of exponent $\alpha$, and its bulk and correlation function again agree across the three laws and with the conjectured limit.

Figure~\ref{fig:finitevar} makes these comparisons more quantitative and estimates the distance between various network simulations and direct simulations of the predicted limit as a function of $n$. The predicted limits are given by self-consistent stochastic equations that were solved numerically using damped Picard iterations, estimating the self-consistent terms of a given iterate from the statistics of $M$ independent copies of the previous iterate equation (either their mean in case 1, their mean and covariance of the Gaussian input in case 2, or constructing a truncated LePage series for heavy-tailed couplings). Distances between a network simulation and the limit were estimated through the Wasserstein distance $W_1$ for light tails and the Kolmogorov--Smirnov distance for heavy tails. For references, we estimated a \emph{sampling level}, which is the distance obtained when the network is replaced by $n$ trajectories of an independent solution of the limit (see details in Appendix~\ref{app:numerics}). 

\begin{figure}
	\includegraphics[width=\textwidth]{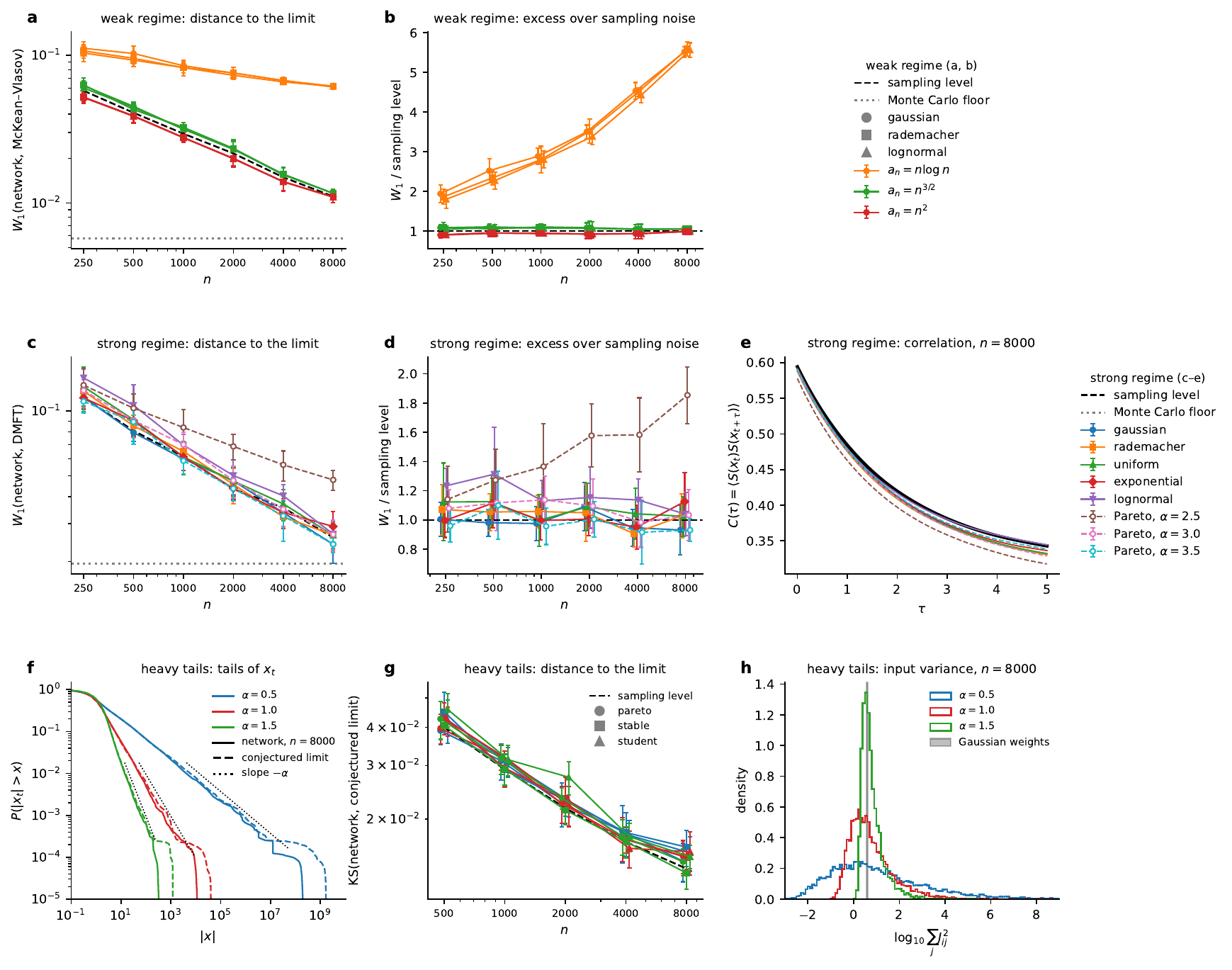}
	\caption{Tests of the conjecture in the three regimes: (a-b) Weakly heterogeneous regimes. (a) Wasserstein distance between the network and the McKean--Vlasov solution, $a_n\in\{n\log n,n^{3/2},n^2\}$ and three laws of the couplings (markers); dashed: sampling level; dotted: Monte Carlo floor, estimating the fluctuation in the estimate of the limit. (b)~Ratio of the distance to the sampling level. (c-e) Strongly heterogeneous regime ($a_n=n$), for five light-tailed laws and symmetric Pareto couplings with $\alpha\in\{5/2,3,7/2\}$ normalized to unit variance; (c)~distance to the Gaussian self-consistent limit; (d)~ratio to the sampling level; (e)~correlation functions at $n=8\,000$ and in the limit. (f-g) {Heavy tails regime} ($J_{ij}=g_{ij}n^{-1/\alpha}$, {$I=\bar J=0$, }$\alpha\in\{1/2,1,3/2\}$): (f)~Tails of $x_t$ for one network of size $n=8\,000$ with Pareto couplings (solid) and for the conjectured limit (dashed); dotted: slope $-\alpha$. (g)~Kolmogorov--Smirnov distance between networks with Pareto, stable and Student couplings {(markers; colors as in f)} and the conjectured limit; {dashed: sampling level.} (h)~Distribution across neurons of the input variance $\sum_jJ_{ij}^2$, $n=8\,000$; gray band: range for Gaussian couplings with $\sigma=2$.}
	\label{fig:lighttails}\label{fig:finitevar}
\end{figure}

\begin{enumerate}
\item \emph{Weakly heterogeneous regime:} We considered $J_{ij}=\bar J/n+\sigma g_{ij}/\sqrt{a_n}$ with $\bar J=-1$, $\sigma=2$, and $g_{ij}$ centered with unit variance for 5 distributions: Gaussian, Rademacher, uniform, centered exponential, centered log-normal {(three of which are shown in Figure~\ref{fig:lighttails}a)}, and three scalings: $a_n=n^{3/2}$, $a_n=n^2$ and $a_n=n\log(n)$. 

In all cases, we found that the distance to the McKean--Vlasov solution decreases as $n$ increases {(Figure~\ref{fig:lighttails}a)}, with no impact of the distribution of the coupling weights on the convergence other than the scaling. Scalings $a_n=n^{3/2}$ and $a_n=n^2$ converged very fast to the limit, with a rate comparable to the sampling level, meaning that already at these sizes, the network is as close to the McKean--Vlasov limit as $n$ independent copies of it {(panel b)}\footnote{It remains within $\pm10\%$ of the ratios obtained with $10$ realizations of the couplings.}.
For $a_n=n\log n$, the Wasserstein distance to the McKean-Vlasov limit also decreases, supporting the conjecture of a convergence, but it remains much larger than in the other two scalings, and in particular much slower than the sampling level, suggesting that the random part of the input vanishes slower than the finite-size effects.

\item \emph{Strongly heterogeneous regime} We next considered $J_{ij}=\bar J/n+\sigma g_{ij}/\sqrt{n}$ for 8 centered distributions of the $g_{ij}$, including lognormal distributions and power-law distributions with tails decaying faster than $1/x^\alpha$ for $\alpha>2$ that are not covered by any theorems. Consistently with the conjecture, we found that all distributions converged to the predicted limit (the dynamical mean-field theory, DMFT), and they were all close to the sampling level {(panels c and d)}\footnote{Between $0.91$ and $1.13$ for the four laws with exponential moments; for the log-normal law, the ratio decreases from $1.31$ to $1.04$ at $n=8\,000$, with power-law distributions remaining farther in particular for $\alpha$ close to {the boundary $\alpha=2$ of the finite-variance class}. For $\alpha=5/2$, the distance to the limit decreases, but more slowly than the sampling {level.}}. Moreover, the correlation functions agree with the limit (panel e), with an error decreasing like $n^{-1/2}$ (Figure~\ref{fig:strongsupp}a). Of note, the time horizon chosen, $T=20$, is much longer than the correlation time{, and no deviation appears at long times}. Figure~\ref{fig:strongsupp}b illustrates self-averaging: the fluctuations of the population average $n^{-1}\sum_iS(x^i_T)$ from one realization of the couplings to another decrease like $n^{-1/2}$. We also note that the largest coupling grows like $n^{2/\alpha-1/2}$, up to about $19$ for $\alpha=5/2$ and $n=8\,000$ (Figure~\ref{fig:strongsupp}d), so that the operator norm of $J$ diverges. {Note that the fourth-moment assumption of Proposition~\ref{prop:OpNorm} below (weakly heterogeneous regime) fails for $\alpha<4$.} The spectral radius of {$J-\bar J/n$}, however, stays at $\sigma=2$ (Figure~\ref{fig:strongsupp}e), in agreement with the fact that the spectral radius of a matrix with i.i.d.\ entries converges under a mere second moment assumption~\cite{bordenave-caputo-chafai-tikhomirov:18}. 
\item \emph{Heavy tail regimes}. Figure~\ref{fig:lighttails}f--h tests numerically the conjecture on heavy-tailed networks, for couplings $J_{ij}=g_{ij}n^{-1/\alpha}$, with symmetric Pareto, $\alpha$-stable or Student $g_{ij}$ normalized to the same tail $\P(|g_{ij}|>x)\sim x^{-\alpha}$ and $\alpha\in\{1/2,1,3/2\}$. We confirmed the consistency of the network distribution with the conjectured limit, both considering the distribution (and noting, in particular, that the tails of $x_t$ decay with exponent $\alpha$ over more than two decades, panel f), and by computing the Kolmogorov--Smirnov distance between the network and the conjectured limit, remaining very close to the sampling level of the theoretical limit for all values of $\alpha$ and all three laws (panel g). Moreover, we analyzed the variance of the input received by neurons, and found, as expected, that it spreads several orders of magnitude (panel h). This suggests that the network may not {be} self-averaging, a point we make more precise in Section~\ref{sec:HeavyTails}: each neuron sees a different environment, and only the population as a whole follows the limit.
\end{enumerate}

\subsection{Mathematical setting}\label{sec:Setting}
{The rigorous results of Sections~\ref{sec:Weak} to~\ref{sec:HeavyTails} use the following setting and notation.}

The class of systems we are interested in have two levels of randomness: networks are heterogeneous, thus associated to a quenched disorder in the coupling coefficients, and moreover the dynamics are stochastic. To account for these two levels of randomness, we consider two probability spaces:

\begin{itemize} 
\item the filtered probability space $\big(\Omega, \mathcal{F}, (\mathcal{F}_t)_{t \in [0,T]}, \P \big)$ associated with the stochastic dynamics: the random initial conditions of the system and the Brownian paths belong to this space. 
\item $\big(\tilde{\Omega}, \mathcal{F}_J, \mathcal{P}_J \big)$ is the probability space associated with the random environment. The random variables $J_{ij}$, $i,j \in [\![1,n]\!]$ are $\big(\tilde{\Omega}, \mathcal{F}_J, \mathcal{P}_J \big)$- random variables. They are said \textit{frozen} in the sense that they do not vary during the time evolution of the dynamics.
\end{itemize}

We will denote by $\Exp$ the expectation over $\P$, and $\E_J$ that over $\mathcal{P}_J$.

Sample trajectories are continuous real functions of $[0,T]$, so that the probability distribution of $X^i_t$ belongs to $\M_1^+(\C)$ the set of probability measures on the $\C:=\C([0,T],\R)$. We work under the following assumptions:
\begin{itemize}
	\item $f$ is a {$K$}-Lipschitz-continuous map, typically $f(x)=-x$. 
	\item $S$ is a {$K$-}Lipschitz-continuous bounded map, typically a sigmoid $S(x)=\tanh(x)$.
\end{itemize}

For any measure $\mu \in \M_1^+(\C)$, level of disorder $\sigma \in \R_+$ and mean coupling $\bar{J}\in\R$, we introduce the following maps respectively defined on $[0,T]^2$ and $[0,T]$:
\begin{equation*}
\begin{cases}
  K_{\mu,\sigma}(s,t)&:=\displaystyle{\sigma^2 \int_{\C} S(x_t)S(x_s) d\mu(x)} ,\\
  m_{\mu}(t) &:= \displaystyle{\bar{J} \int_{\C} S(x_t)d\mu(x)}.
\end{cases}
\end{equation*}

We shall consider the system with iid initial conditions with law $\mu_0$ a $\big(\Omega, \mathcal{F}, \P \big)$ random variable. The stochastic dynamics of the $n$-neuron network are thus given by the equation:
\begin{equation}\label{eq:InteractingSDE}
\left\{
  \begin{array}{ll}
   dX^i_t = \Big(f(X^i_t) + \sum_{j=1}^n J_{ij} S(X^j_t)\Big) dt + dW^i_t \\
   (X^i_0)_{i=1, \ldots, n} \eqlaw \mu_0^{\otimes n},
  \end{array}
\right.
\end{equation} 
{(that is, $\lambda=1$ in~\eqref{eq:network}),} where the $J_{ij}$ are independent random variables (i.i.d.\ unless stated otherwise) with mean $\frac{\bar{J}}{n}$ and variance $\frac{\sigma^2}{a_n}$, where $(a_n)$ is a diverging sequence. To split the contributions of average coupling and fluctuations, it will be useful to denote $\hat{J}_{ij} := \sigma^{-1}\sqrt{a_n} \big(J_{ij}-\frac{\bar{J}}{n}\big)$ for general distributions of $J_{ij}$; these are centered with unit variance, so that $J_{ij}=\frac{\bar J}{n}+\frac{\sigma}{\sqrt{a_n}}\hat J_{ij}$, and the level of disorder $\sigma$ always appears explicitly outside the sums. When the distribution of the $J_{ij}$ is Gaussian, we will rather denote this centered variable $\tilde{J}_{ij}$.

It is a classical result that \eqref{eq:InteractingSDE} admits a unique weak solution (or a strong solution when the initial conditions are given by prescribed random variable). We will denote by $Q^{a_n}_{\sigma}(\hat{J})$ this solution restricted to the $\sigma$-algebra $\G_T=\sigma(X^i_s, 0 \leq s \leq T, 1 \leq i \leq n)$, or $Q^{a_n}_{\sigma}(\tilde{J})$ for Gaussian interactions. These laws represent the \emph{quenched} behaviors: they depend on the realization of the disorder, and are probability measures on $\C^n$, and when they are well-defined, we will note the averaged laws $\hat{Q}^{a_n}_{\sigma}:=\E_J\big(Q^{a_n}_{\sigma}(\hat{J})\big)$ and $\tilde{Q}^{a_n}_{\sigma}:=\E_J\big(Q^{a_n}_{\sigma}(\tilde{J})\big)$. {When $a_n=n$, we simply write $\hat Q^n_\sigma$ and $\tilde Q^n_\sigma$.} 

\section{Weakly heterogeneous couplings}\label{sec:Weak}
In this section, the couplings are $J_{ij}=\frac{\bar J}{n}+\frac{\sigma}{\sqrt{a_n}}\hat J_{ij}$ with $n=o(a_n)$ and $\hat J_{ij}$ i.i.d, centered couplings with unit variance. {The conjecture states that the empirical measure $\hat\mu_n$ converges to $Q_0$}, where $Q_0$ is the solution of the implicit mean-field (McKean--Vlasov) equation:
\begin{equation}\label{eq:MVlimit}
\left\{
  \begin{array}{ll}
   d\bar{X}_t = \Big( f(\bar{X}_t) + m_{Q_0}(t) \Big) dt + {dW_t} \\
   \bar{X}_0 \eqlaw \mu_0
  \end{array}
\right.
\end{equation}

\subsection{Literature review}\label{sec:WeakLit}
Along with the original conjecture, {Geman~\cite{geman:82}} proved a number of results in the context of system~\eqref{eq:network}  with $\lambda=0$. {In the weakly heterogeneous regime,} {he} demonstrated that the system converges towards the deterministic dynamical system 
\[\frac{dx}{dt}=-x + \bar{J} S(x)\]
as soon as the coefficients $J_{ij}$ are iid with mean $\bar{J}/n$ and with a distribution having an analytic Fourier transform at zero (i.e., only depending on the first moment regardless of the details of their distribution). The degenerate case $\sigma=0$, in which all couplings equal $\bar J/n$, is the classical exchangeable setting of propagation of chaos~\cite{sznitman:89}{, with large deviations of the empirical measure described in~\cite{dawson-gartner:87,budhiraja-dupuis-fischer:12}}. Recent {breakthroughs} also significantly improved estimates of the quantitative convergence rate. For exchangeable systems, coupling and entropy methods~\cite{sznitman:89,jabin-wang:17} estimated a relative entropy of order $k/n$ between the law of $k$ particles and $Q^{\otimes k}$, but Lacker~\cite{lacker:23} derived the optimal order $(k/n)^2$. The method relied on a BBGKY hierarchy of relative entropies. This convergence rate further holds uniformly in time under functional inequalities~\cite{lacker-leflem:23} and for nonlinear dependence on the measure~\cite{arnese-lacker:26}.

{Many authors have pursued the study of the weakly heterogeneous regime successfully, and there was important recent progress on the problem in the past decade when the connections are random graphs. In the context of general couplings (that is, not specific to graphs), Dai~Pra and den~Hollander~\cite{daipra-denhollander:96} considered mean-field interactions depending on an i.i.d.\ random medium attached to the particles; the McKean--Vlasov limit is then a family of coupled equations indexed by the medium. In this framework, the authors also described the large deviations and Gaussian fluctuations around the limit. In spatially extended networks of neurons, where the coupling is random with a law depending on the positions of the two neurons, propagation of chaos holds in the neural-field limit~\cite{touboul:14} where the limit is indexed by space. The case of singular spatial weights together with a random environment was studied in~\cite{lucon-stannat:14}, and multiscale singular coupling inspired by neural connectivity in~\cite{touboul-multiscale}. For general frozen synaptic weights possibly correlated with random transmission delays, propagation of chaos then holds for almost every realization of the weights when their law does not depend on the postsynaptic neuron, and on average over the weights otherwise (for interactions that do not depend on the state of the postsynaptic neuron)~\cite{quininao-touboul:15}. For deterministic multi-agent systems, mean-field and graph limits were derived for influence weights attached to the agents and evolving in time~\cite{ayi-pouradierduteil:21}, and for general index-dependent interactions, by treating the index as an additional variable~\cite{paul-trelat:24}; see~\cite{ayi-pouradierduteil:24} for a review.}

Results relative to graphs have seen remarkable advances in the past decade. Dense graph limits were introduced in~\cite{medvedev:14} using the theory of graphons~\cite{lovasz:12}. Quenched convergence to the McKean--Vlasov limit for Erd\H{o}s--R\'enyi graphs that are not too sparse was studied in~\cite{delattre-giacomin-lucon:16}, and inhomogeneous, time-varying random graphs in~\cite{bhamidi-budhiraja-wu:19}. Laws of large numbers and large deviations for networks with diverging average degree graphs were derived in~\cite{coppini-dietert-giacomin:20, oliveira-reis:19}, and quenched results on inhomogeneous graphs were covered in~\cite{lucon:20}. {Dense graph limits also led to graphon mean-field systems~\cite{bayraktar-chakraborty-wu:23}.} {In the language of this paper, most of these works concern couplings $J_{ij}=\xi_{ij}/(np_n)$, where $\xi_{ij}\in\{0,1\}$ are the entries of the adjacency matrix of a graph with edge density $p_n$ (or nonnegative graphon weights), the sign of the interaction being carried by the interaction function. For an Erd\H{o}s--R\'enyi graph, $J_{ij}$ has mean $1/n$ and variance $(1-p_n)/(n^2p_n)$, that is, $\bar J=1$, $\sigma=1$ and $a_n=n^2p_n/(1-p_n)$: the weakly heterogeneous regime $n=o(a_n)$ is exactly the regime of diverging average degree $np_n\to\infty$, in which the convergence to the McKean--Vlasov limit is proved~\cite{delattre-giacomin-lucon:16,oliveira-reis:19}, this condition being optimal~\cite{oliveira-reis:19}. Item~1 of the conjecture is thus established for Bernoulli couplings. What these works do not address is its universality with respect to the law of the couplings: i.i.d.\ couplings with a general, possibly signed, distribution with mean $\bar J/n$ and variance $\sigma^2/a_n$.}

Part of this problem is handled by~\cite{jabin-poyato-soler:25}, where a general theory of mean-field limits for non-exchangeable systems is developed which covers sparse graphs and asymmetric weights. The framework of~\cite{jabin-poyato-soler:25} allows general weights, but requires the row and column sums of $|J_{ij}|$ to be bounded and $\max_{i,j}|J_{ij}|\to0$. For i.i.d.\ signed couplings $J_{ij}=\bar J/n+\sigma\hat J_{ij}/\sqrt{a_n}$, since $\sum_j|J_{ij}|\geq\sigma a_n^{-1/2}\sum_j|\hat J_{ij}|-|\bar J|$ and $\frac1n\sum_j|\hat J_{ij}|\to\E_J|\hat J_{11}|>0$ almost surely, these conditions require $\liminf_n a_n/n^2>0$. Conversely, when $\liminf_n a_n/n^2>0$ and the $\hat J_{ij}$ are bounded, they hold, and~\cite{jabin-poyato-soler:25} gives quenched convergence. Limits of sequences of weighted directed graphs, dense, sparse or of intermediate density were studied in~\cite{kuehn-xu:22} for deterministic dynamics. Lacker's hierarchy method yields sharp bounds for non-exchangeable systems with $\max_i\sum_j|J_{ij}|\leq1$~\cite{lacker-yeung-zhou:24}; for i.i.d.\ signed couplings with variance $\sigma^2/a_n$, this condition requires $\liminf_n a_n/n^2>0$, for the same reason. {For non-exchangeable systems with dense enough interaction matrices, the fluctuations around the limit were also shown to coincide with those of the exchangeable case~\cite{shkolnikov-yeung:26}.}

{In summary, quenched convergence to the McKean--Vlasov limit was known for unweighted random graphs that are not too sparse, and for i.i.d.\ signed couplings when $\liminf_n a_n/n^2>0$ (under boundedness or moment assumptions on the $\hat J_{ij}$); to our knowledge, none of these results covers i.i.d.\ signed couplings with $n=o(a_n)$ and $a_n=o(n^2)$.} {Where they apply, these coupling arguments also give the quenched statements of part~(B) of the conjecture (Section~\ref{sec:Conjecture}): the law of a fixed neuron given the couplings converges to the McKean--Vlasov law, and propagation of chaos holds given the couplings.}

\subsection{\texorpdfstring{Quenched convergence for $n=o({a_n})$}{Quenched convergence for n=o(a\_n)}}\label{sec:WeakNew}
We extend these results to show that the weakly heterogeneous regime is, on finite time intervals, a simple perturbation of the mean-field case, controlled by the operator norm of the random part of the coupling matrix, which will allow us to expand the universality results. For this result, we use the $p$-Wasserstein distance on  $\M_1^+(\C)$:
\[
d^V_{T,p}(\mu,\nu):=\inf_{\xi\in\mathcal C_{\mu,\nu}}\Big(\int_{\C^2}\sup_{t\leq T}|x_t-y_t|^p\,d\xi(x,y)\Big)^{\frac1p},
\]
for $p\geq1$ and where $\mathcal C_{\mu,\nu}$ is the set of all couplings of $(\mu,\nu)$, that is, the set of probability measures on $\C^2$ with marginals $\mu$ and $\nu$. It is finite on measures with $\int\sup_{t\leq T}|x_t|^pd\mu(x)<\infty$, and convergence for $d^V_{T,p}$ is equivalent to weak convergence together with convergence of the $p$-th moments $\int\sup_{t\leq T}|x_t|^p\,d\mu$~\cite{villani:09}. We work under the assumptions of the paper that $f$ and $S$ are $K$-Lipschitz and $S$ is bounded. 

\begin{proposition}\label{prop:OpNorm}
Assume that $\int|x|^2d\mu_0(x)<\infty$ and $\sup_{n,i,j}\E_J\hat J_{ij}^4<\infty$. Then:
\begin{enumerate}
\item there exists a constant $C_T$ depending on $T$, $K$, $\Vert S\Vert_\infty$ and $\bar J$ such that, for every realization of the couplings and of the noise,
\[
d^V_{T,2}\big(\hat\mu_n,\hat\mu_n^{\mathrm{mf}}\big)\leq C_T \frac{\sigma}{\sqrt{a_n}}\, \Vert \hat J\Vert_{\mathrm{op}},
\]
where $\hat\mu_n^{\mathrm{mf}}$ is the empirical measure of the network with constant couplings $ J_{ij}=\bar J/n$, driven by the same initial conditions and Brownian motions;
\item $\E_J\Vert \hat{J} \Vert_{\mathrm{op}}\leq C \sqrt{n}$, with $C$ depending only on $\sup_{n,i,j}\E_J\hat J_{ij}^4$.
\end{enumerate}
Consequently, if $n=o(a_n)$, then $d^V_{T,2}(\hat\mu_n,Q_0)\to0$ in probability for every $T>0$, where $Q_0$ is the law of the McKean--Vlasov process~\eqref{eq:MVlimit}. For a sequence with consistent couplings (that is, when the $n$-neuron coupling matrices are the upper-left $n\times n$ submatrices of the same doubly infinite i.i.d.\ array, and the initial conditions and Brownian motions are the first $n$ elements of a fixed i.i.d. sequence), the convergence is almost sure, for almost every sequence of coupling matrices (quenched convergence). 
\end{proposition}
{This result is quenched: the convergence holds in probability with respect to the couplings, and almost surely when the coupling matrices are nested.} It holds on any time interval, covers the whole range $n=o(a_n)$ only under a fourth moment, which improves existing coupling arguments based on exponential moments that typically require $n^2=O(a_n)$. Our argument is {analogous} to the arguments used for interacting diffusions on Erd\H{o}s--R\'enyi graphs with diverging degrees~\cite{delattre-giacomin-lucon:16,coppini-dietert-giacomin:20,oliveira-reis:19}{; to the best of our knowledge, it had not been written for general i.i.d.\ signed couplings, which} are not covered by the non-exchangeable framework of~\cite{jabin-poyato-soler:25} either, since the row sums $\sum_j|J_{ij}|\sim n/\sqrt{a_n}$ diverge when $a_n=o(n^2)$. Its proof is given in Appendix~\ref{app:ProofWeak}. 
The universality in this regime only uses the fact that the random part of the coupling matrix is small in operator norm\footnote{In that view, it is natural that $a_n=n$ serves as a threshold because, for $a_n=n$, {$\frac{\sigma}{\sqrt n}\Vert\hat J\Vert_{\mathrm{op}}\to2\sigma$ (for i.i.d.\ entries with a finite fourth moment~\cite{yin-bai-krishnaiah:88})} does not vanish.}, and this whatever the law of its entries, so the bound holds for every realization, with a constant $C_T$ that depends neither on the couplings nor on their law.

{The same coupling estimate gives the quenched statements of part~(B) of the conjecture for fixed neurons: under the assumptions of Proposition~\ref{prop:OpNorm}, $d\big(Q^1_n(J),Q_0\big)\to0$ and $d\big(Q^{12}_n(J),Q_0\otimes Q_0\big)\to0$ in probability with respect to the couplings (see the end of Appendix~\ref{app:ProofWeak}).}

We can relax the condition on the fourth moment and replace it by a condition slightly stronger than a second moment, but at the cost of {almost sure} quenched convergence. For a truncation level $K_n>0$, let
\[
\varepsilon_n(K_n):=\max_{i,j\leq n}\E_J\big[\hat J_{ij}^2\,\mathbf 1_{|\hat J_{ij}|>K_n}\big].
\]
\begin{proposition}\label{prop:Truncation}
Assume that $\int|x|^2d\mu_0(x)<\infty$, $n=o(a_n)$, and that there exist $K_n>0$ such that
\begin{equation}\label{eq:TruncCond}
K_n\sqrt{\log n}=o\big(\sqrt{a_n}\big)\qquad\text{and}\qquad n\,\varepsilon_n(K_n)=o\big(K_n\sqrt{a_n}\big).
\end{equation}
Then, for every $T>0$, $d^V_{T,2}(\hat\mu_n,{Q_0})\to0$ in probability. 

\noindent We add that condition~\eqref{eq:TruncCond} holds:
\begin{enumerate}
\item[(i)] without any further assumption, if $a_n/(n\sqrt{\log n})\to\infty$;
\item[(ii)] for every $a_n$ with $n=o(a_n)$, if $M:=\sup_{n,i,j}\E_J\big[\hat J_{ij}^2\sqrt{\log(e+|\hat J_{ij}|)}\big]<\infty$.
\end{enumerate}
\end{proposition}
{The proof (Appendix~\ref{app:ProofWeak}) relies on a truncation argument. The coupling coefficients are truncated at level $K_n$ for each $n$, the truncated part is controlled in operator norm using sharp bounds on the norm of random matrices with bounded independent entries~\cite{bandeira-vanhandel:16}, and the large coupling terms are {controlled} through the row sums of their absolute values.

\subsection{Open problems}\label{sec:WeakOpen}
These two results bring us much closer to completing the program engaged in the conjecture, and only boundary problems remain open, including:
\begin{itemize}
\item Proposition~\ref{prop:Truncation} gives convergence in probability under a second moment when $a_n/(n\sqrt{\log n})\to\infty$, and in the whole regime $n=o(a_n)$ under the condition $\E_J[\hat J_{ij}^2\sqrt{\log(e+|\hat J_{ij}|)}]$ bounded. These conditions look mostly technical, although we have not been able to lift them yet. The convergence under a second moment was only proven when $n\sqrt{\log n}=o(a_n)$, and the factor $\sqrt{\log n}$ arose from the norm bound for matrices with bounded entries in Bandeira and van Handel~\cite{bandeira-vanhandel:16}. The latter bound is sharp in the general class of matrices considered in~\cite{bandeira-vanhandel:16}, but the operator we need to control for our purpose is of a particular kind (with truncated couplings), and stronger estimates may hold for these. Second, almost sure convergence was only proven {for nested i.i.d.\ couplings with a law independent of $n$ and a finite fourth moment}, and improving these conditions remains an open problem. 
\item For $n=o(a_n)$, Proposition~\ref{prop:OpNorm} gives the rate $\sqrt{n/a_n}$ for the empirical measure, and Proposition~\ref{prop:Truncation} gives no rate. In the exchangeable case much sharper rates are known: the $k$-neuron marginals are at distance $O((k/n)^2)$ from the product law in relative entropy~\cite{lacker:23}, uniformly in time under functional inequalities~\cite{lacker-leflem:23}, and these rates extend to non-exchangeable couplings whose row sums are controlled~\cite{lacker-yeung-zhou:24}, which for i.i.d.\ signed couplings requires $\liminf_n a_n/n^2>0$. The optimal rate, its dependence on $a_n$ and on the law of the couplings, and quantitative propagation of chaos for $k$-neuron marginals, are open problems for the scalings and weights of Proposition~\ref{prop:OpNorm}.
\end{itemize}

\section{Strongly heterogeneous couplings}\label{sec:StrongHet}
\subsection{Literature review}\label{sec:StrongLit}
{In the strongly heterogeneous regime, Geman and Hwang} proved for $S(x)=x$~\cite{geman:1982b} that when the coefficients $J_{ij}=\frac{\sigma}{\sqrt{n}}\xi_{ij}$ where $\xi_{ij}$ iid centered random variables with unit variance (with no other assumption on their distribution), the system converges weakly towards independent copies of a Gaussian process with prescribed mean and covariance structure, given by an implicit equation, and only depending on $\sigma$.
On the strongly heterogeneous regime, Geman~\cite{geman:82} also formulates the conjecture that the latter result holds for nonlinear map $S$ and general interaction coefficients. These questions long remained elusive, and few rigorous results were proved. Thirteen years after the introduction of Geman and Hwang conjecture, a similar result was proved in the context of stochastic interacting spin-glass model with Gaussian coupling $J_{ij}$~\cite{ben-arous-guionnet:95,ben-arous-guionnet:97}. This approach was extended to prove quenched propagation of chaos in the same model in~\cite{guionnet:97}, and related results were derived for Glauber dynamics~\cite{grunwald:96} and extended to describe long-time behavior and aging in spherical models~\cite{ben-arous-dembo-guionnet:01,ben-arous-dembo-guionnet:06}, or applied to neural systems with spatial interactions, delays and state-dependent interactions~\cite{cabana-touboul:15,cabana-touboul:16}. Technically, the principle of the proof consists in demonstrating that the sequence of empirical measures satisfies a large-deviations principle, and many arguments rely on the Gaussian nature of the coupling weights. In particular, this proves the Gaussian case of Geman's conjecture, extended to the case of networks with stochastic inputs, with a convergence quantified through exponential tightness estimates of the empirical measure and a weak Large Deviation Principle (LDP)~\cite{cabana-touboul:15},  with a good rate function $H_{\sigma}$ that attains its minimal value at a unique probability measure $Q_{\sigma}$, in an average sense, given by the process:
\begin{equation}\label{eq:GaussianLimit}
\left\{
  \begin{array}{ll}
   d\bar{X}_t = \Big( f(\bar{X}_t) + G^{Q_{\sigma}}_t + m_{Q_{\sigma}}(t) \Big) dt + dW_t \\
   \bar{X}_0 \eqlaw \mu_0,
  \end{array}
\right.
\end{equation}
where $\big(G^{Q_{\sigma}}_t, 0\leq t \leq T\big)$ is a centered Gaussian process with covariance $K_{Q_{\sigma},\sigma}${, independent of $(\bar X_0,W)$} (see also~\cite{faugeras-tanre:25}), consistent with the conjecture. 

On the physics side, equation~\eqref{eq:GaussianLimit} is often referred to as the Dynamic Mean-Field Theory, DMFT. It was discussed in the landmark work of Sompolinsky, Crisanti and Sommers~\cite{sompolinsky-crisanti-etal:88} where it was interpreted as the dynamics of a particle in a self-consistent potential, and then reformulated as a large-deviation principle for the empirical measure in~\cite{vanmeegen-kuhn-helias:21} (see~\cite{helias-dahmen:20} for a textbook account). Universality for strongly heterogeneous interactions is less developed. In physics, it is largely admitted (see the citation of M\'ezard, Parisi and Virasoro in the introduction).

In mathematics, results relative to universality are scarce. A tour de force in the domain was achieved by Moynot and Samuelides in 2002 for discrete-time versions of system~\eqref{eq:network}, where it was shown that for specific couplings the networks remained asymptotically arbitrarily close to the Gaussian case, so that the large-deviation principle persists in the system and that it implies universality for non-Gaussian weights with Gaussian tails~\cite{moynot-samuelides:02}. However, extending these results to continuous-time diffusions or to weaker assumptions has remained open for many years. A natural hypothesis would be that these results can be extended to continuous-time dynamics by discretizing equation~\eqref{eq:network} on increasingly fine time grids. However, when doing so naively, the estimates of Moynot and Samuelides blow up, and no control is conserved in the continuous-time limit. We propose in Theorem~\ref{thm:Convergence4} an alternative extension that conserves the continuous-time nature of the system. Using a distinct set of tools, important rigorous results for the universality were obtained in recent years. In particular, Dembo, Lubetzky and Zeitouni~\cite{dembo-lubetzky-zeitouni:21} showed universality of the limit for confined Langevin dynamics with \emph{linear} interactions $\frac{\beta}{\sqrt n}\sum_jJ_{ij}X^j_t$ where the $J_{ij}$ are iid centered random variables with unit variance and exponential tails. These results are very strong: they ensure that for any $\beta$ the empirical measure on a fixed (but arbitrary) time interval converges almost surely to the same limit as in the Gaussian case. Their approach consists in freezing the interaction on a time grid allows them to compare Gaussian and non-Gaussian frozen systems through the multivariate Lindeberg principle proposed by Chatterjee~\cite{chatterjee:06} (this tool also yields universality of the free energy of the Sherrington--Kirkpatrick model~\cite{carmona-hu:06}, see~\cite{talagrand:11} for the statics). These estimates are made on the event where the operator norm of the coupling matrix is bounded, where they can then control the frozen and true dynamics by a pathwise coupling argument rather than by a comparison of densities. Another remarkable work towards universality is due to Dembo and Gheissari~\cite{dembo-gheissari:21}, where they show that averaged observables of diffusions interacting \emph{linearly} through a random matrix with independent entries depend only on \emph{the first two moments} of the entries, assuming sub-exponential tails for the couplings. Here again the method of proof is distinct from all previous approaches: it relies on stochastic Taylor expansion and moment matching, and this result covers a large part of the conjecture on strongly heterogeneous synapses in the case of \emph{linear} interactions. Finally, an alternative method was proposed in~\cite{celentano-cheng-montanari:21} based on time-discretizations and reduction to iterative schemes known as Approximate Message Passing (AMP), which yields universality with respect to the law of i.i.d.\ sub-Gaussian entries (see also~\cite{gerbelot-etal:24} for stochastic gradient descent, and~\cite{dudeja-lu-sen:23} for the universality of AMP). To our knowledge, universality for nonlinear interactions $\sum_jJ_{ij}S(X^j)$ in continuous time, mentioned as open in~\cite[Remark~1.4]{dembo-lubetzky-zeitouni:21}, had not been established. 

In terms of the universality of the propagation of chaos property, averaged and quenched results were proven for Gaussian couplings for a closely related Langevin dynamics of spin glasses, in which the interaction $\sum_jJ_{ij}x^j$ is linear and the nonlinearity lies in the confining drift~\cite{ben-arous-guionnet:95,guionnet:97}. For the model~\eqref{eq:network}, in which the interaction acts through the nonlinear map $S$, the quenched statements of part~(B) of the conjecture (Section~\ref{sec:Conjecture}) are, to our knowledge, open, even for Gaussian couplings.

\subsection{{Continuous-time universality with sub-Gaussian tails}}\label{sec:StrongNew}

We extend the class of models proven to be in the universality class by showing an averaged convergence for sub-Gaussian couplings $\hat J_{ij}$ i.i.d., centered, with unit variance and sub-Gaussian moments:
\begin{equation}\label{hyp:sub-Gaussian-tail}
\ \sup_{n,i,j}\mathcal E_J\big[e^{\theta_0\hat J_{ij}^2}\big]\leq D\tag{$H_J$}
\end{equation}
for some $\theta_0,D>0$.}

\begin{theorem}\label{thm:Convergence4}
Under the current assumptions, suppose that the couplings satisfy~\eqref{hyp:sub-Gaussian-tail}  and that $\mu_0$ has a finite second moment. Then, for every $T>0$ and every $\epsilon>0$, there exists $C>0$ such that, for all $n\geq2$, with $B(Q_\sigma,\epsilon)$ the open ball for a metric of weak convergence on $\M_1^+(\C)$ (e.g., the L\'evy--Prokhorov metric),
\[
\hat Q^n_\sigma\big(\hat\mu_n\notin B(Q_\sigma,\epsilon)\big)\leq\frac{C}{\log n}.
\]
In particular, $\big(\hat{Q}^n_{\sigma}(\hat{\mu}_{n} \in \cdot ) \big)_{n \in \N^*}$ converges towards $\delta_{Q_{\sigma}} \in \M_1^+\big( \M_1^+(\C ) \big)$.
\end{theorem}

This result is averaged over the couplings, but it holds on every time interval and for the nonlinear interaction of~\eqref{eq:InteractingSDE}, which~\cite{dembo-lubetzky-zeitouni:21,dembo-gheissari:21} do not cover. Its complete proof is given in Appendix~\ref{sec:NonGaussian}, and uses relative entropy estimates to handle a technical difficulty arising when applying~\cite{moynot-samuelides:02} to continuous time processes that limited the time horizon of these results in~\cite{cabana2016large}. The principle of the proof consists in comparing the averaged laws with non-Gaussian and Gaussian couplings through their Radon-Nikodym derivatives with respect to the law of independent neurons. After averaging, these densities factorize over the rows of the coupling matrix, and each row factor can be compared with its Gaussian counterpart by a Lindeberg replacement, control the relative entropy $H(\hat Q^n_\sigma|\tilde Q^n_\sigma)$, and conclude using Donsker--Varadhan's inequality. This argument is likely not optimal for computing a rate of convergence, yielding a rate equal to $1/\log n$ instead of an exponential rate. A direct extension of our result should allow obtaining a rate $(\log n)^{-k}$ for every $k\geq1$ (by considering moments higher than order 4 in Lemma~\ref{lem:Bmoments}), but our estimates do not provide a polynomial or exponential rate of convergence.}

\subsection{Open problems}\label{sec:StrongOpen}
These results leave a wide range of open problems mathematically. In particular, quenched universality in the nonlinear interactions case is still open; Theorem~\ref{thm:Convergence4} covers the averaged case, with a rate $1/\log n$, and its proof requires independent (non-symmetric) couplings. Extending the Lindeberg-and-coupling strategy of~\cite{dembo-lubetzky-zeitouni:21} seems the natural route to quenched statements. Moreover, the conditions used in most universality theorems are stronger than what is claimed in the conjecture. The results of~\cite{dembo-lubetzky-zeitouni:21} require exponential tails, and Theorem~\ref{thm:Convergence4} sub-Gaussian tails. We are thus far from the minimal assumptions that are included under the conjecture, that would just require some finite small degree polynomial moments. A variety of couplings considered in our numerical experiments tend to confirm that the conjecture should hold on these cases, including log-normal weights (that have moments of all orders but no exponential moments) or power-law couplings (Pareto $\alpha\in\{5/2, 3,7/2\}$) which have moments of order more than 2, but not every polynomial moment bounded. In the latter case where $\P(|g_{ij}|>x)\sim x^{-\alpha}$ with $\alpha\in(2,4)$ and $J_{ij}=g_{ij}/\sqrt n$, while the central limit theorem holds row by row, the network contains a vanishing fraction, on the order of $n^{2-\alpha/2}$ in number, of couplings of order one, and its largest coupling grows like $n^{2/\alpha-1/2}$. These rare strong couplings control the operator norm of the coupling matrix, and for symmetric couplings, its largest eigenvalues~\cite{auffinger-benarous-peche:09}, and may therefore affect long-time stability and the transition to chaos, although they affect a vanishing fraction of neurons on finite time intervals.
Finally, while the fluctuations around the McKean--Vlasov limit were recently shown to be universal with respect to the structure of dense interaction matrices~\cite{shkolnikov-yeung:26}, the strongly heterogeneous analogue of that result is open to the best of our knowledge.

\section{Heavy-tailed couplings}\label{sec:HeavyTails}
\subsection{Literature review}\label{sec:HeavyLit}
For heavy-tailed couplings, much less is known mathematically. Recent interest arose in the past five years in the physics community. Discrete-time networks with Cauchy weights scaled by $1/n$ were studied in~\cite{kusmierz-ogawa-toyoizumi:20}, where authors derived a recursion relationship on the Cauchy population average of $|S(x_j)|$, which they use to accurately predict a continuous transition between a quiescent phase and an active phase. They also analyze the dynamics and find that it propagates almost discretely, depending only on a few of the largest weights, forming a propagation akin to a critical branching process, and therefore in turn produces scale-free avalanches. Stationary states of rate networks with $\alpha$-stable weights, $\alpha\in(1,2)$ were studied in~\cite{wardak-gong:21,wardak-gong:22}. These states are analogs of a time-independent version of the limit conjectured in Section~\ref{sec:HeavyTails}. The authors show that eigenvectors of the Jacobian are multifractal, that individual neurons' autocorrelation fails to self-average, and derive generalized autocorrelation functions in the frequency domain. Finally, finite-size effects associated with the transition from quiescence to chaos were studied in~\cite{xie-mihalas-kusmierz:25}. 

On the mathematical side, little remains known about these systems, and we are not aware of any formulation of a dynamical equation for the system before this paper. Outside of the dynamics of neural networks with heavy-tailed couplings, a number of mathematical works in the past two decades have characterized in various situations the spectrum of heavy-tailed random matrices~\cite{ben-arous-guionnet:08,bordenave-caputo-chafai:11} (see also~\cite{soshnikov:04,bordenave-guionnet:13}). In the context of machine learning, it was shown that feedforward networks with infinite-width and stable weights are stable processes~\cite{favaro-fortini-peluchetti:20}, see also~\cite{jung-lee-lee-yang:23} for heavy- and light-tailed weights. The statics of a spin glass with heavy-tailed couplings, the L\'evy spin glass, were studied rigorously in~\cite{chen-kim-sen:23}. To the best of our knowledge, however, no formulation of a possible limit is available for recurrent stochastic dynamics with heavy-tailed couplings and no characterization for the law of the trajectories on finite time intervals was proposed, and \emph{a fortiori}, no characterization of possible propagation of chaos and self-averaging discussed in part~(B) of the conjecture, have been studied mathematically.

\subsection{Heuristics}\label{sec:HeavyHeuristics}
We work under the assumptions that $J_{ij}=g_{ij}/n^{1/\alpha}$ with $(g_{ij})$ be i.i.d. with regularly varying tails, $\P(|g_{ij}|>x)\sim x^{-\alpha}$, $\alpha\in(0,2)$, and we assume for simplicity that the distribution of $g_{ij}$ is symmetric. Collections of heavy tailed random variables often present vastly distinct property compared to, e.g., random variables with exponential moments. In particular, 
\begin{enumerate}
\item For each $\varepsilon>0$, the number of incoming weights of a given neuron with $|J_{ij}|>\varepsilon$ converges in law to a Poisson variable with mean of order $\varepsilon^{-\alpha}$, while the total contribution of the weights below $\varepsilon$ has variance of order $\varepsilon^{2-\alpha}$. The interaction is therefore largely (that is, with an arbitrarily small error) carried by a \emph{sparse} weighted graph with finite mean degree.
\item The quantity $\sum_j J_{ij}^2$, which in the Gaussian case concentrates around $\sigma^2$ and sets the variance of the local field, converges in law to a positive $\alpha/2$-stable variable. We therefore do not expect self-averaging to happen, and different neurons see local fields with different statistics, even as $n\to\infty$.
\item Some couplings will be extremely large, with the largest coupling in the whole network diverging as $n^{1/\alpha}$, which will affect regularity of the solutions and their dynamics.
\end{enumerate}
Our conjecture proposes that the system converges to the self-consistent equation:
\[
\begin{gathered}dx_t=\big(f(x_t)+\mu\,\Exp[S(x_t)]+Z^x_t\big)\,dt+\lambda\,dW_t,\\ Z^x_t=C_\alpha\sum_{k\geq1}\varepsilon_k\Gamma_k^{-1/\alpha}S(y^k_t),\quad y^k\overset{\text{i.i.d.}}{\sim}Q=\mathrm{Law}(x).\end{gathered}
\]
The LePage representation of $Z_t^x$ can be seen as an illustration of item~1 above. Neuron $i$ receives large inputs from presynaptic neurons $j_{1},j_{2},\ldots$ chosen uniformly at random, with weights close to $\varepsilon_k\Gamma_k^{-1/\alpha}$. Because the weights are independent and asymmetric, neuron $i$ is, with high probability, not among the large-weight presynaptic neurons of $j_{k}$, and the graph of large weights explored from $i$ converges locally to a Poisson weighted infinite tree~\cite{aldous-steele:04,bordenave-caputo-chafai:11}. On a directed tree, the presynaptic neurons are driven by disjoint subtrees; they should therefore be asymptotically independent, independent of neuron $i$, and distributed according to $Q$, which is precisely the self-consistency in the LePage series. In this picture, molecular chaos fails pairwise along the large-weight edges, but in a way that is invisible at the level of the empirical measure.

This heuristic underlies part~(B) of our conjecture, that the empirical measure $\hat\mu_n$ converges to $Q$ in probability, but in contrast with the Gaussian case, the law of a fixed neuron \emph{given the couplings} does not converge to $Q$, but to a random law determined by the local limit of the weighted graph around that neuron ($Q$ is then recovered only after averaging over the disorder). The latter point could be the mathematical counterpart of the non-self-averaging reported by Wardak and Gong~\cite{wardak-gong:22}, and is consistent with the trajectories of Figure~\ref{fig:universality} (bottom), which are different between coupling matrices even though the statistics of the population agree. It also suggests that the natural framework is that of local weak convergence for sparse networks of interacting processes~\cite{lacker-ramanan-wu:23}, and the marginal dynamics on random trees developed in~\cite{lacker-ramanan-wu:23b,ramanan:23}, rather than classical mean-field theory.

\subsection{Well-posedness of the conjectured limit}\label{sec:HeavyWellPosed}
As a starting point on the conjecture, we study whether the putative limit of the system is well-posed here. In detail, we show existence for all $\alpha\in (0,2)$, and uniqueness for $\alpha\in[1,2)$ and when $f$ is affine. The argument is not a usual Picard iterate using a Lipschitz constant of $Q\mapsto Z^Q$, but it relies on the scaling of the LePage series and their couplings. 

Fix $T>0$, $\alpha\in(0,2)$, $\lambda\geq0$, $\mu\in\R$, an initial law $\mu_0$ on $\R$. We continue to assume $f$ and $S$ $K$-Lipschitz continuous and $S$ bounded. Let $(\Gamma_k)_{k\geq1}$ be the arrival times of a unit-rate Poisson process and $(\varepsilon_k)_{k\geq1}$ i.i.d.\ Rademacher signs. For $Q\in\M_1^+(\C)$, let $(y^k)_{k\geq1}$ be i.i.d.\ with law $Q$, and assume that $(\Gamma_k)$, $(\varepsilon_k)$, $(y^k)$, $x_0\sim\mu_0$ and the Brownian motion $W$ are independent. Set
\[
Z^Q:=C_\alpha\sum_{k\geq1}\varepsilon_k\Gamma_k^{-1/\alpha}S(y^k),\qquad m_Q(t):=\mu\,\Exp_Q[S(y_t)].
\]

Building all processes based on the same random variables is a key ingredient of our proof, together with the following result, which relies on some classical estimates (see e.g.~\cite{ledoux-talagrand:91}), but that we include for completeness and to introduce our notations, and include a brief proof in Appendix~\ref{app:ProofsHeavy}. 
\begin{lemma}\label{lem:LePageTail}
Let $H$ be a separable Hilbert space and $(v_k)$ i.i.d.\ $H$-valued random variables, independent of $(\Gamma_k)$ and $(\varepsilon_k)$, with $m_\alpha:=\Exp\Vert v_1\Vert^\alpha<\infty$. Then $X:=\sum_k\varepsilon_k\Gamma_k^{-1/\alpha}v_k$ converges almost surely in $H$, its law only depends on the law of $v_1$, and
\[
\forall u>0,\qquad \P(\Vert X\Vert>u)\leq\frac{2}{2-\alpha}\,m_\alpha\,u^{-\alpha}.
\]
Consequently, with $\psi(u):=u\big(1+\log^+(1/u)\big)$,
\[
\Exp\big[(\Vert X\Vert\wedge1)^\alpha\big]\leq\psi\Big(\frac{2}{2-\alpha}\,m_\alpha\Big).
\]
\end{lemma}

Applied to $H=L^2([0,T])$ and $v_k=S(y^k)$, with $\Vert v_k\Vert\leq\sqrt T\Vert S\Vert_\infty$, Lemma~\ref{lem:LePageTail} shows that $Z^Q$ is a well-defined $L^2([0,T])$-valued random variable whose law only depends on $Q$, and that
\begin{equation}\label{eq:ZTailUniform}
\P\big(\Vert Z^Q\Vert_{L^2(0,T)}>u\big)\leq\frac{2C_\alpha^\alpha}{2-\alpha}\big(T\Vert S\Vert_\infty^2\big)^{\frac\alpha2}u^{-\alpha}\qquad\text{uniformly in }Q.
\end{equation}
In particular $Z^Q\in L^1([0,T])$ almost surely and, $f$ being Lipschitz, the equation
\begin{equation}\label{eq:HeavyLimitFixedQ}
x_t=x_0+\int_0^t\big(f(x_s)+m_Q(s)\big)ds+\int_0^tZ^Q_sds+\lambda W_t
\end{equation}
has a unique continuous solution, pathwise. We denote by $\Phi(Q)$ its law on $\C$. A \emph{solution of the conjectured limit equation} is a fixed point $Q=\Phi(Q)$.

\begin{theorem}\label{thm:HeavyWellPosed}
\begin{enumerate}
\item[(i)] For every $\alpha\in(0,2)$, the map $\Phi$ has at least one fixed point.
\item[(ii)] If $\alpha\in[1,2)$ and $f$ is affine, $f(x)=-\kappa x+c$, the fixed point is unique.
\end{enumerate}
\end{theorem}
The proof of Theorem~\ref{thm:HeavyWellPosed} is given in Appendix~\ref{app:ProofsHeavy}. The conjectured limit thus exists for every $\alpha\in(0,2)$, and is unique for $\alpha\in[1,2)$ when $f$ is affine; the simulations of Section~\ref{sec:Numerics} show that the network statistics agree with it. The convergence of the network to this limit remains to be proved.

\subsection{Open problems and possible strategies}\label{sec:HeavyOpen}
The conjecture is mathematically largely open. First, it seems relatively clear that heavy-tailed couplings challenge most of the existing approaches reviewed for the weakly or strongly heterogeneous cases. The approach based on large deviations for strongly heterogeneous couplings fails immediately, because it relies on integrating the couplings out of the Girsanov density $\E_J\exp\{\int G\,dW-\frac12\int G^2\}$, but these terms are not integrable for heavy-tailed couplings. Any possible large-deviation principle for these systems would be vastly distinct, with a different speed and a rate function of a different nature, because most deviations are due to a few large couplings (possibly similar to what has been established for the spectrum of heavy-tailed random matrices~\cite{bordenave-caputo:14}).

In the case of strongly heterogeneous couplings, assuming linear interactions was another method to make the problem more tractable~\cite{dembo-lubetzky-zeitouni:21}. However, it seems to us that for heavy-tailed couplings, the result of universality strongly relies on the boundedness of $S$, because the norm of a heavy-tailed coupling matrix, and for symmetric couplings its largest eigenvalues, grow with $n$~\cite{ben-arous-guionnet:08,bordenave-caputo-chafai:11} so that solutions may grow, on a fixed time interval, faster and faster as $n$ increases, with no bound uniform in $n$.

Coupling methods used for weakly heterogeneous synapses also present difficulties in the heavy-tailed setting. In particular, the Gronwall arguments of Appendix~\ref{app:ProofWeak} involve the operator norm of the coupling matrix and the row sums $\sum_j|J_{ij}|$ of its large entries, which would not be bounded uniformly in $n$.

Finally, well-posedness of the self-consistent equation is also unclear: a natural Lipschitz constant of $Q\mapsto Z^Q$ is $\sum_k\Gamma_k^{-1/\alpha}$, which is a.s. finite only for $\alpha<1$ and has no finite mean, so that the usual Wasserstein contraction arguments do not apply directly. This difficulty can be partly circumvented by exploiting the stable scaling of the LePage series instead of a Lipschitz constant (Section~\ref{sec:HeavyWellPosed}), and completing the proof of uniqueness is open. 

\medskip

As done in the weakly heterogeneous setting, a possible approach could be to use truncation arguments. In detail, a method could split the couplings at a level $\varepsilon>0$. The couplings larger than $\varepsilon$ form a sparse directed graph with Poisson$(\varepsilon^{-\alpha})$ in-degrees, which converges locally to a Poisson weighted infinite tree; at fixed $\varepsilon$, the theory of interacting diffusions on sparse graphs and their local limits should apply~\cite{lacker-ramanan-wu:23,lacker-ramanan-wu:23b}. The couplings smaller than $\varepsilon$ produce a stable input whose scale could vanish as $\varepsilon\to0$ from Lemma~\ref{lem:LePageTail}. The same lemma controls the convergence of the truncated self-consistent equation to the full one. This would give the convergence of the averaged law of a single neuron, and then of the empirical measure, since two given neurons have disjoint large-weight neighborhoods with high probability.

Finally, in cases where existence and uniqueness of the limit is assured, an interesting approach would be to explore the large-weight neighborhood of a neuron up to a given depth, and couple it with the Poisson weighted infinite tree. This approach would require controlling estimates of the error, with costs possibly similar to the term $\Exp[(|\delta x_t|\wedge1)^\alpha]$ that comes out in the proof of Theorem~\ref{thm:HeavyWellPosed}.

\section{Discussion}\label{sec:Discussion}
Forty years after Geman and Hwang, the McKean--Vlasov class for weakly heterogeneous couplings is, on finite time intervals, essentially established, and the Gaussian self-consistent class for strongly heterogeneous couplings is largely established for linear interactions, but only partially for the nonlinear interactions of neural networks. The heavy-tailed class constitutes a new frontier for this problem. In detail, for weakly heterogeneous couplings, the convergence of the empirical measure to the McKean--Vlasov limit is now established for every scaling $n=o(a_n)$, in probability with respect to the couplings under a fourth moment, almost surely for nested i.i.d.\ arrays (Proposition~\ref{prop:OpNorm}), and in probability under slightly more than a second moment (Proposition~\ref{prop:Truncation}). {The same estimate also gives the quenched statements of part~(B) for fixed neurons.} We expect that proving almost sure convergence under weaker moments is mostly technical, and improving rates of convergence constitute exciting avenues in particular in the perspective of the remarkable recent developments~\cite{lacker:23,lacker-leflem:23,arnese-lacker:26,lacker-yeung-zhou:24}. 

For strongly heterogeneous couplings, the universality is now established for linear interactions, in the quenched setting with exponential tails~\cite{dembo-lubetzky-zeitouni:21}, and for averaged observables~\cite{dembo-gheissari:21}, and we prove it here for nonlinear interactions of neural networks with independent sub-Gaussian couplings (Theorem~\ref{thm:Convergence4}); averaged propagation of chaos follows by exchangeability. The quenched statements of part~(B) for nonlinear interactions are open, even for Gaussian couplings, as are couplings with heavier tails but finite variance (log-normal or power-law, for which the simulations of Section~\ref{sec:Numerics} support the conjecture), symmetric or correlated couplings, and the fluctuations around the limit.

For heavy-tailed couplings, the mathematical formulation of the conjecture is original to this paper, and we have provided a conjectured limit given by a self-consistent equation driven by an $\alpha$-stable process. We have also shown that the limit has a solution for every $\alpha\in(0,2)$, and this solution is unique for $\alpha\in[1,2)$ when $f$ is affine (Theorem~\ref{thm:HeavyWellPosed}). The convergence of the network to this limit, averaged or quenched, is open, as are uniqueness for $\alpha<1$ or nonlinear $f$, and the description of the random quenched law of a single neuron predicted by part~(B), but our extensive numerical simulations support the conjecture. Finally, in all regimes, all rigorous results hold on fixed time intervals, and questions related to uniform-in-time convergence, the order of the limits $n\to\infty$ and $t\to\infty$ remain open to the best of our knowledge. 

The conjecture of Section~\ref{sec:Conjecture} covers couplings whose row sums obey a law of large numbers, a central limit theorem, or a generalized central limit theorem. This leaves out a whole world of scalings that support distinct regimes of behavior. An example which has received attention recently is that of strongly coupled networks. Contrasting with all the regimes discussed in the universality conjecture where couplings are scaled for the total input $\sum_jJ_{ij}S(x^j_t)$ to remain bounded, these networks are such that
\[
J_{ij}=\gamma(n)\,\frac{g_{ij}}{n},\qquad \gamma(n)\xrightarrow[n\to\infty]{}\infty,
\]
so that the law of large numbers no longer necessarily produces a finite interaction term: the mean input $\gamma(n)\,\bar g\,\Exp[S(x_t)]$ would diverge, except if $\Exp[S(x_t)]=0$. This regime lies outside the Geman--Hwang classification. In computational neuroscience, a central paper~\cite{vanvreeswijk-sompolinsky:96} (see also~\cite{vanvreeswijk-sompolinsky:98,brunel:00,renart-etal:10,rosenbaum-doiron:14,barral-reyes:16}) suggested that in these regimes the network quickly converges, under specific conditions, on a regime where the inputs are balanced, corresponding to $\Exp[S(x_t)]=0$. Mathematically, questions related to this regime were studied in \cite{quininao-touboul:20} for a particular example, and then generalized and formalized in~\cite{quininao-touboul:25}, both in the case of deterministic couplings. Under appropriate regularity and stability conditions, the system collapses at very fast timescales on a balance manifold, which is a set of measures where the interaction vanishes. Subsequent works relevant to this problem include the derivation of a nonlocal reaction--diffusion FitzHugh--Nagumo system from the mesoscopic kinetic equation {(see~\cite{mischler-quininao-touboul:16} for this kinetic model)} in the strong local interaction limit using relative entropy~\cite{crevat-faye-filbet:19},  explicit, uniform-in-time Wasserstein estimates~\cite{blaustein-filbet:22}, strong and quantitative convergence for spatially extended networks~\cite{blaustein:22}, or $L^\infty$ estimates~\cite{blaustein-bouin:22}.

Another regime that is largely open would be the regime $a_n=o(n)$ for couplings with finite variance, in which the fluctuations of the input diverge:
\[J_{ij}=\frac{\bar{J}}{n} + \frac{\gamma(n)}{\sqrt{n}}\,g_{ij},\qquad \gamma(n)\xrightarrow[n\to\infty]{}\infty.\] 
To our knowledge, it has not been studied.

Sparse graphs with bounded average degree also lie outside the conjecture, although $a_n$ is of order $n$ for them, because their row sums are asymptotically Poisson rather than Gaussian: the particles that are effectively interacting remain correlated in the limit, breaking the sto\ss zahlansatz. Local weak limits can however be derived~\cite{lacker-ramanan-wu:21,lacker-ramanan-wu:23,lacker-ramanan-wu:23b,ganguly-ramanan:24}, see~\cite{ramanan:23} for a survey. 

Forty years after its formulation, the program of Geman and Hwang has thus seen important advances, and keeps raising new questions.

{\bf Acknowledgments and AI statement.} The author is deeply grateful to Tanguy Cabana, whose PhD thesis~\cite{cabana2016large} laid the groundwork for the strongly heterogeneous part of this paper. Many of the ideas developed here grew out of our fun discussions and of his heroic preliminary work, and it was a pleasure to think about these questions together. The numerical simulations were first developed in custom Matlab code. The author then used the assistance of Anthropic's Claude (Opus 5.5) to translate this code to Python and optimize it to reduce its runtime. This made it possible to run the simulations presented here in a reasonable time (about 10 hours of parallel computations on a MacBook Pro (M5)). Claude was also used throughout the paper to improve the writing, find appropriate references, validate proofs, and propose refinements of some of the mathematical statements and of their proofs.

\appendix
\setcounter{figure}{0}
\renewcommand\thefigure{S.\arabic{figure}}

\bigskip 
\centerline{\Large \bf Appendix}
\section{Weakly Heterogeneous Weights}\label{app:ProofWeak}
\begin{proof}[Proof of Proposition~\ref{prop:OpNorm}]
This proof relies on controlling the distance between the network equation and the mean-field system through the operator norm of the random coupling part (in the spirit of coupling arguments for random graphs~\cite{delattre-giacomin-lucon:16,oliveira-reis:19}), which in turn is controlled using Lata{\l}a's inequality on the operator norm of matrices with independent centered entries~\cite{latala:05} and bounded fourth moments. 

{Throughout, $\Vert\cdot\Vert$ is the Euclidean norm on $\R^n$, $Q_0$ is the law of the McKean--Vlasov process~\eqref{eq:MVlimit}, and the systems are coupled: for every $n$, the initial conditions and Brownian motions of the $n$ neurons are the first $n$ elements of a fixed i.i.d.\ sequence $(X^i_0,W^i)_{i\geq1}$, independent of the couplings.} We use coupling techniques to compare the empirical measure $\hat \mu_n$ of the randomly coupled network to a mean-field limit, and introduce two coupled systems that make the comparisons easier: the empirical measure $\hat{\mu}_n^{mf}$ of a network with deterministic coupling:
\[
dX^{\mathrm{mf},i}_t=\Big(f(X^{\mathrm{mf},i}_t)+\frac{\bar J}{n}\sum_{j=1}^nS(X^{\mathrm{mf},j}_t)\Big)dt+dW^i_t,\qquad X^{\mathrm{mf},i}_0=X^i_0,
\]
and then compare the latter to the empirical measure of $n$ independent copies of the McKean--Vlasov process~\eqref{eq:MVlimit}, $\bar\mu_n$, defined through the coupling:
\[
d\bar X^i_t=\big(f(\bar X^i_t)+m_{{Q_0}}(t)\big)dt+dW^i_t,\qquad \bar X^i_0=X^i_0,\qquad m_{{Q_0}}(t)=\bar J\,\Exp\big[S(\bar X^1_t)\big].
\]
Pairing the $i$-th coordinates and using the Cauchy--Schwarz inequality, for any two such systems $\mathbf Y,\mathbf Z$ driven by the same initial conditions and Brownian motions, so that the noise cancels and $\mathbf Y-\mathbf Z$ is continuously differentiable in time, with $\mathbf Y_0-\mathbf Z_0=0$,
\begin{equation}\label{eq:OpNormPairing}
d^V_{T,2}\Big(\frac1n\sum_i\delta_{Y^i},\frac1n\sum_i\delta_{Z^i}\Big)^2\leq\frac1n\sum_{i=1}^n\sup_{t\leq T}|Y^i_t-Z^i_t|^2\leq\frac Tn\int_0^T\Big\Vert\frac{d}{dt}(\mathbf Y_t-\mathbf Z_t)\Big\Vert^2dt.
\end{equation}

Let $\Delta:=\mathbf X-\mathbf X^{\mathrm{mf}}$. Since $J=\frac{\bar J}{n}+\frac{\sigma}{\sqrt{a_n}}\hat J$ and both systems are driven by the same noise, the noise cancels and we have:
\[
\frac{d}{dt} \left(X^i_t-X^{mf,i}_t\right) =f(X^i_t)-f(X^{\mathrm{mf},i}_t)+\frac{\bar J}{n}\sum_{j=1}^n\big(S(X^j_t)-S(X^{\mathrm{mf},j}_t)\big)+\frac{\sigma}{\sqrt{a_n}}\big(\hat J\,S(\mathbf X_t)\big)_i .
\]
Since $f$ and $S$ are Lipschitz continuous and $S$ is bounded, we get $\Vert\dot\Delta_t\Vert\leq c\Vert\Delta_t\Vert+b$ with $b:= \sqrt{n}\,\frac{\sigma}{\sqrt{a_n}}\Vert S\Vert_\infty\Vert  \hat{J}\Vert_{\mathrm{op}}$ and $c=K (1+|\bar J|)$. Gronwall's lemma and $\Delta_0=0$ give $\Vert\Delta_t\Vert\leq bte^{ct}$, hence $\sup_{t\leq T}\Vert\dot\Delta_t\Vert\leq b(1+cT)e^{cT}\leq be^{2cT}$, and by~\eqref{eq:OpNormPairing}
\begin{equation}\label{eq:OpBound}
d^V_{T,2}\big(\hat\mu_n,\hat\mu_n^{\mathrm{mf}}\big)\leq \frac{T\sigma}{\sqrt{a_n}} \,\Vert S\Vert_\infty \,\Vert \hat{J}\Vert_{\mathrm{op}}e^{2cT},
\end{equation}
for every realization of the couplings and of the noise, proving the first point of the proposition. 

Classical results apply to the comparison of the constant coupling network and independent copies of the McKean--Vlasov process, which we briefly outline here. Thanks to the coupling of initial conditions and Brownian motions, the noise cancels again. 
\[
\begin{aligned}\frac{d}{dt} \left(X^{\mathrm{mf},i}_t-\bar X^i_t\right)&=f(X^{\mathrm{mf},i}_t)-f(\bar X^i_t)+\frac{\bar J}{n}\sum_{j=1}^n\big(S(X^{\mathrm{mf},j}_t)-S(\bar X^j_t)\big)\\&\quad+\frac{\bar J}{n} \sum_{j=1}^n\big(S(\bar X^j_t)-\Exp[S(\bar X^1_t)]\big)\end{aligned}
\]
and using Gronwall's lemma, similarly to the above estimates, yields:
\[
d^V_{T,2}\big(\hat\mu_n^{\mathrm{mf}},\bar\mu_n\big)\leq T|\bar J|e^{2cT}\sup_{t\leq T} \left | \frac1 n \sum_{j=1}^nS(\bar X^j_t)-\Exp[S(\bar X^1_t)]\right |=:T|\bar J|e^{2cT}\sup_{t\leq T} \vert R_n(t)\vert, 
\]
and classical results assure that the right hand side vanishes almost surely when $n\to \infty$ (e.g., by application of a functional strong law of large numbers~\cite{mourier:53} or a standard quadratic control). 

We are now in a position to conclude the proof. The triangular inequality gives:
\begin{equation}\label{eq:triang}
d^V_{T,2}\big(\hat\mu_n,{Q_0} \big)\leq d^V_{T,2}\big(\hat\mu_n,\hat\mu_n^{mf})+ d^V_{T,2}\big(\hat\mu_n^{mf},\bar{\mu}_n)+ d^V_{T,2}\big(\bar{\mu}_n,{Q_0}\big).
\end{equation}
The last term vanishes almost surely as $n\to\infty$. Indeed, we know that since $\bar X^j$ are i.i.d.\ with law ${Q_0}$, $\bar\mu_n\to {Q_0}$ weakly almost surely (Varadarajan's theorem). Moreover, noting that $\bar X^1$ is square integrable, the strong law of large numbers implies that $\int\sup_t|x_t|^2d\bar\mu_n\to\Exp[\sup_t|\bar X^1_t|^2]$ almost surely. Finally, since weak convergence and convergence of second moments is equivalent to convergence in $d^V_{T,2}$~\cite{villani:09}, $d^V_{T,2}(\bar\mu_n,{Q_0})\to0$ almost surely.

The last two terms of~\eqref{eq:triang} therefore do not depend on the couplings and tend to $0$ when $n\to \infty$ $\P$ almost surely. By~\eqref{eq:OpBound}, the first term is dominated by the operator norm of $\hat{J}$, and Lata{\l}a's inequality for matrices with independent centered entries~\cite{latala:05} gives
\[
\E_J\Vert\hat J\Vert_{\mathrm{op}}\leq C\Big(\max_i\Big(\sum_j\E_J\hat J_{ij}^2\Big)^{\frac12}+\max_j\Big(\sum_i\E_J\hat J_{ij}^2\Big)^{\frac12}+\Big(\sum_{i,j}\E_J\hat J_{ij}^4\Big)^{\frac14}\Big)\leq C\sqrt n,
\]
implying that $\E_J d^V_{T,2}\big(\hat\mu_n,\hat\mu_n^{mf})\leq C_T \sqrt{n/a_n}$, which tends to $0$ when $n=o(a_n)$, so $d^V_{T,2}\big(\hat\mu_n,\hat\mu_n^{mf})$ tends to $0$ in $\mathcal P_J$ probability. Altogether, we find that $d^V_{T,2}(\hat\mu_n,{Q_0})\to0$ in probability under the joint law of the couplings and the noise. Precisely, on the event $\{\Vert \sigma \hat{J}/\sqrt{a_n}\Vert_{\mathrm{op}}\leq\delta\}$, whose $\mathcal P_J$-probability tends to $1$ for every $\delta>0$, one has $d^V_{T,2}(\hat\mu_n,{Q_0})\leq C_T\delta+r_n$, where $r_n$ does not depend on the couplings and tends to $0$, $\P$-almost surely. If the $\hat J_{ij}$ are the entries of an infinite i.i.d.\ array with a law independent of $n$, the coupling matrices being the upper-left $n\times n$ submatrices of a doubly infinite array $(\hat J_{ij})_{i,j\geq1}$ of i.i.d.\ random variables, then $\Vert\hat J\Vert_{\mathrm{op}}/\sqrt n\to2$ almost surely~\cite{yin-bai-krishnaiah:88}, hence $\Vert \frac{\sigma}{\sqrt{a_n}}\hat J \Vert_{\mathrm{op}}\to0$, and the convergence holds $\P$-almost surely for $\mathcal P_J$-almost every sequence of coupling matrices (quenched convergence).

{Finally, the quenched statements for fixed neurons follow from the same estimates. The averaged law of the network is exchangeable, so that, by~\eqref{eq:OpNormPairing}, the proof of~\eqref{eq:OpBound} and the concavity of $x\mapsto x\wedge1$,
\[
\begin{aligned}\E_J\Exp\Big[\sup_{t\leq T}|X^{1}_t-X^{\mathrm{mf},1}_t|^2\wedge1\Big]&\leq\E_J\Exp\Big[\Big(\frac1n\sum_{i=1}^n\sup_{t\leq T}|X^{i}_t-X^{\mathrm{mf},i}_t|^2\Big)\wedge1\Big]\\&\leq\E_J\Big[\Big(C_T\frac{\sigma^2}{a_n}\Vert\hat J\Vert_{\mathrm{op}}^2\Big)\wedge1\Big]\xrightarrow[n\to\infty]{}0,\end{aligned}
\]
and similarly $\Exp\big[\sup_{t\leq T}|X^{\mathrm{mf},1}_t-\bar X^1_t|^2\wedge1\big]\to0$, independently of the couplings. Since $\bar X^1\sim Q_0$ is independent of the couplings, $\E_J\,d\big(Q^1_n(J),Q_0\big)\to0$, and the same argument applied to $(X^1,X^2)$ and $(\bar X^1,\bar X^2)$ gives $\E_J\,d\big(Q^{12}_n(J),Q_0\otimes Q_0\big)\to0$.}
\end{proof}

{\begin{proof}[Proof of Proposition~\ref{prop:Truncation}]
We use the notation of the proof of Proposition~\ref{prop:OpNorm}, and resume the proof from the inequality $d^V_{T,2}(\hat\mu_n,{Q_0})\leq d^V_{T,2}(\hat\mu_n,\hat\mu^{\mathrm{mf}}_n)+r_n$ with $r_n:=T|\bar J|e^{2cT}\sup_{t\leq T}|R_n(t)|+d^V_{T,2}(\bar\mu_n,{Q_0})$, which does not depend on the couplings and tends to $0$, $\P$-almost surely. Under a second moment condition only, the convergence of $d^V_{T,2}(\hat\mu_n,\hat\mu^{\mathrm{mf}}_n)$ is slightly more subtle, and we proceed by a truncation argument. 

In detail, using the arguments from the proof of Proposition~\ref{prop:OpNorm}, noting $\Gamma:=\sup_{t\leq T}\Vert {\tfrac{\sigma}{\sqrt{a_n}}\hat J}\,S(\mathbf X_t)\Vert$, we get by a simple application of Gronwall's lemma that $\Vert \mathbf X-\mathbf X^{\mathrm{mf}}\Vert\leq\Gamma te^{ct}$, so that $\sup_{t\leq T}\Vert\frac{d}{dt}(\mathbf X-\mathbf X^{\mathrm{mf}})\Vert\leq\Gamma(1+cTe^{cT})\leq\Gamma e^{2cT}$, and, by~\eqref{eq:OpNormPairing},
\[
d^V_{T,2}\big(\hat\mu_n,\hat\mu_n^{\mathrm{mf}}\big)\leq Te^{2cT}\,\frac{\Gamma}{\sqrt n}
\]
for every realization of the couplings and of the noise.

We then truncate the weights, and define
\[
A_{ij}:=\hat J_{ij}\mathbf 1_{|\hat J_{ij}|\leq K_n}-\E_J\big[\hat J_{ij}\mathbf 1_{|\hat J_{ij}|\leq K_n}\big],\qquad B_{ij}:=\hat J_{ij}\mathbf 1_{|\hat J_{ij}|>K_n}-\E_J\big[\hat J_{ij}\mathbf 1_{|\hat J_{ij}|>K_n}\big].
\]
so that (given that $\hat J_{ij}$ are centered), $\hat J=A+B$. For $v=S(\mathbf X_t)$, we control $\Vert \hat J v\Vert$ using this operator decomposition. We have $\Vert Av\Vert\leq\sqrt n\Vert S\Vert_\infty\Vert A\Vert_{\mathrm{op}}$, and $|(Bv)_i|\leq\Vert S\Vert_\infty\rho_i$ with $\rho_i:=\sum_{j}|B_{ij}|$, thus:
\[
d^V_{T,2}\big(\hat\mu_n,\hat\mu_n^{\mathrm{mf}}\big)\leq\Lambda_n:=Te^{2cT}\sigma\Vert S\Vert_\infty\bigg(\frac{\Vert A\Vert_{\mathrm{op}}}{\sqrt{a_n}}+\Big(\frac1{na_n}\sum_{i=1}^n\rho_i^2\Big)^{1/2}\bigg),
\]
which depends only on the couplings. We show that both terms tend to $0$ in $\mathcal P_J$-probability.

For the truncated part, we will use Bandeira and van Handel's sharp nonasymptotic bounds for the norm of $A$. We note that its coefficients, $A_{ij}$, are independent, centered random variables bounded by $2K_n$, and $\E_JA_{ij}^2\leq\E_J\hat J_{ij}^2\leq1$. Bandeira and van Handel~\cite{bandeira-vanhandel:16} provide sharp estimates in particular for symmetric matrices with bounded entries. To apply those results, we introduce a symmetric $2n\times2n$ matrix $D:=\left(\begin{smallmatrix}0&A\\A^\top&0\end{smallmatrix}\right)$, whose norm is also $\Vert A\Vert_{\mathrm{op}}$, and that has independent centered entries on and above the diagonal, and its parameters $\tilde\sigma:=\max_k\big(\sum_l\E_JD_{kl}^2\big)^{1/2}$ and $\tilde\sigma_*:=\max_{k,l}\Vert D_{kl}\Vert_\infty$ satisfy $\tilde\sigma\leq\sqrt n$ and $\tilde\sigma_*\leq2K_n$. By~\cite[Corollary~3.12 and Remark~3.13]{bandeira-vanhandel:16} with $\varepsilon=1/2$, there is a universal constant $\tilde c$ such that, for all $t\geq0$,
\[
\mathcal P_J\big(\Vert A\Vert_{\mathrm{op}}\geq3\sqrt2\,\sqrt n+t\big)\leq2n\,e^{-t^2/(4\tilde cK_n^2)}.
\]
With $t=2K_n\sqrt{2\tilde c\log(2n)}$, we obtain that, with $\mathcal P_J$-probability at least $1-(2n)^{-1}$,
\[
\frac{\Vert A\Vert_{\mathrm{op}}}{\sqrt{a_n}}\leq3\sqrt2\sqrt{\frac n{a_n}}+2\sqrt{2\tilde c}\,\frac{K_n\sqrt{\log(2n)}}{\sqrt{a_n}},
\]
and the right-hand side tends to $0$ by $n=o(a_n)$ and~\eqref{eq:TruncCond}.

For the control of the term arising with large couplings, we note $\eta_{ij}:=|\hat J_{ij}|\mathbf 1_{|\hat J_{ij}|>K_n}\leq\hat J_{ij}^2\mathbf 1_{|\hat J_{ij}|>K_n}/K_n$, thus $\E_J\eta_{ij}\leq\varepsilon_n(K_n)/K_n$ and $\E_J\eta_{ij}^2\leq\varepsilon_n(K_n)$. Moreover $|B_{ij}|\leq\eta_{ij}+\E_J\eta_{ij}$, so that $\rho_i\leq\sum_j\eta_{ij}+n\varepsilon_n(K_n)/K_n$. By independence of the entries of a row,
\[
\E_J\Big[\Big(\sum_j\eta_{ij}\Big)^2\Big]=\sum_j\E_J\eta_{ij}^2+\sum_{j\neq k}\E_J\eta_{ij}\,\E_J\eta_{ik}\leq n\varepsilon_n(K_n)+\frac{n^2\varepsilon_n(K_n)^2}{K_n^2},
\]
hence $\E_J\rho_i^2\leq2n\varepsilon_n(K_n)+4n^2\varepsilon_n(K_n)^2/K_n^2$ and
\[
\E_J\Big[\frac1{na_n}\sum_{i=1}^n\rho_i^2\Big]\leq2\varepsilon_n(K_n)\,\frac n{a_n}+4\bigg(\frac{n\,\varepsilon_n(K_n)}{K_n\sqrt{a_n}}\bigg)^2.
\]
The first term tends to $0$ since $\varepsilon_n\leq1$ and $n=o(a_n)$, and the second by assumption~\eqref{eq:TruncCond}. Markov's inequality then gives $\frac1{na_n}\sum_i\rho_i^2\to0$ in $\mathcal P_J$-probability, so $\Lambda_n\to0$ in $\mathcal P_J$-probability, which proves the first part of the proposition.

Finally, we confirm that the sufficient conditions provided in the proposition are valid. (i) Since $\varepsilon_n\leq1$, the choice $K_n:=n^{1/2}(\log n)^{-1/4}$ gives $K_n\sqrt{\log n}/\sqrt{a_n}=(n\sqrt{\log n}/a_n)^{1/2}$ and $n\varepsilon_n(K_n)/(K_n\sqrt{a_n})\leq(n\sqrt{\log n}/a_n)^{1/2}$, which both tend to $0$ when $a_n/(n\sqrt{\log n})\to\infty$. (ii) Let $K_n:=(n/\log n)^{1/2}$. On $\{|\hat J_{ij}|>K_n\}$, $\sqrt{\log(e+|\hat J_{ij}|)}\geq\sqrt{\log(e+K_n)}$, so that $\varepsilon_n(K_n)\leq M/\sqrt{\log(e+K_n)}$, and we get $K_n\sqrt{\log n}/\sqrt{a_n}=(n/a_n)^{1/2}\to0$. Moreover, for $n\geq3$, $\log(e+K_n)\geq\log K_n=\frac12(\log n-\log\log n)\geq\frac14\log n$, since $x/2-\log x\geq1-\log2>0$ for all $x>0$; hence $\varepsilon_n(K_n)\leq2M(\log n)^{-1/2}$ and $n\varepsilon_n(K_n)/(K_n\sqrt{a_n})\leq2M(n/a_n)^{1/2}\to0$.
\end{proof}}

\section{Proof of Theorem~\ref{thm:Convergence4}}\label{sec:NonGaussian}\label{app:ProofStrong}

Throughout this section, $J_{ij}=\frac{\bar J}{n}+\frac{\sigma}{\sqrt n}\hat J_{ij}$, where the $\hat J_{ij}$ satisfy~\eqref{hyp:sub-Gaussian-tail}, {$f$ and $S$ are $K$-Lipschitz}, that $S$ is bounded (and, without loss of generality, up to a rescaling of $\sigma$, assume that $\Vert S\Vert_\infty\leq1$) and $\mu_0$ has a finite second moment. We denote by $\hat Q^n_\sigma$ the averaged law of the network on $\C^n$, $\C:=C([0,T],\R)$, by $\tilde J_{ij}$ i.i.d.\ standard Gaussian variables, and by $\tilde Q^n_\sigma$ the averaged law of the network with couplings $\frac{\bar J}{n}+\frac{\sigma}{\sqrt n}\tilde J_{ij}$. Let $P$ be the law on $\C$ of the {uncoupled process}
\begin{equation}\label{eq:Uncoupled}
dX_t={f(X_t)}\,dt+dW_t,\qquad X_0\sim\mu_0,
\end{equation}
and, for $x\in\C$,
\begin{equation}\label{def:W}
W_t(x):=x_t-x_0{-\int_0^tf(x_s)\,ds},
\end{equation}
which is the Brownian motion driving~\eqref{eq:Uncoupled} under $P$. Throughout, $C$ denotes a constant depending only on $T,\sigma,\bar J,K$ and on the constants $\theta_0,D$ of~\eqref{hyp:sub-Gaussian-tail} (and on $\epsilon$ in the conclusion), which may change from line to line. The principle of the proof relies on showing that the Radon-Nikodym derivatives of the $n$-neurons network with general weights is arbitrarily close to that of the network with Gaussian weights, and to this end we use exact expressions of the densities and an original argument based on their relative entropy. These estimates, together with an exponential tightness property, provide the convergence result. 

\paragraph{\emph{Preliminary estimates on sub-Gaussian random variables:}}

We start by a few Gaussian calculus results that will be useful for the proof, and that are summarized in the following lemma. The proofs rely on~\cite{vershynin:18} (specifically, Prop.~2.5.2 in particular) together with elementary calculus. 

Consider $X$ a centered random variable with sub-Gaussian tails $\E e^{\theta_0X^2}\leq D$ ($\E$ denotes the expectation, because we have in mind the couplings here) and $g\sim N(0,1)$. These are referred to as preliminary formula (1-4) hereafter. 
\begin{enumerate}
\item There is $\kappa=\kappa(\theta_0,D)\geq\frac12$ such that $\E e^{sX}\leq e^{\kappa s^2}$ and $\E e^{sg}\leq e^{\kappa s^2}$ for all $s\in\R$.
\item If $Z_1,\dots,Z_m$ are independent, each satisfying (1), and $\beta\in\R^m$, then $Y=\sum_k\beta_kZ_k$ satisfies $\E e^{sY}\leq e^{\kappa vs^2}$ for all $s$, with $v=\Vert\beta\Vert^2$.
\item If $\E e^{sY}\leq e^{ws^2}$ for all $s$, then $\E e^{\lambda Y^2}\leq(1-4w\lambda)^{-1/2}$ for $0\leq\lambda<\frac1{4w}$, and $\E|Y|^p\leq C_pw^{p/2}$.
\item For $k\geq0$ and $b\geq0$, $\E\big[|Z|^ke^{b|Z|}\big]\leq C_ke^{2\kappa b^2}$ for $Z\in\{X,g\}$.
\end{enumerate}

Moreover, by~\cite[Thm.~4.4.5]{vershynin:18}, there is $C_0$ depending only on the sub-Gaussian norm of the entries such that $\Vert\hat J\Vert_{\mathrm{op}}\leq C_0(2\sqrt n+t)$ with probability at least $1-2e^{-t^2}$. This implies that, with $t=s\sqrt n$, $\mathcal P_J(\Vert\hat J\Vert_{\mathrm{op}}/\sqrt n>C_0(2+s))\leq2e^{-s^2n}\leq2e^{-s^2}$, so that for all $\lambda>0$, 
\[\sup_n\E\exp\big(\lambda\Vert\hat J\Vert_{\mathrm{op}}/\sqrt n\big)<\infty.\] 
The same holds for Gaussian couplings, but also $\E\Vert\tilde J\Vert_{\mathrm{op}}\leq2\sqrt n$, by~\cite[Thm.~2.13]{davidson-szarek:01}. Gaussian concentration thus implies that $\mathcal P_J\big(\Vert\tilde J\Vert_{\mathrm{op}}\geq(2+s)\sqrt n\big)\leq e^{-s^2n/2}$ for all $s\geq0$.

\medskip

The rest of the proof relies on estimates of the densities of $\hat Q^n_\sigma$ and $\tilde Q^n_\sigma$ with respect to $P^{\otimes n}$, which have an explicit structure. For paths $\mathbf x\in\C^n$, outside a $P^{\otimes n}$-null set, let
\[
A^i_j:=\int_0^TS(x^j_t)\,dW_t(x^i),\qquad M_{jk}:=\int_0^TS(x^j_t)S(x^k_t)\,dt,
\]
and let $J_i=(J_{i1},\dots,J_{in})$ be the $i$-th row of the coupling matrix. By Girsanov's theorem,
\[
\frac{dQ(J)}{dP^{\otimes n}}(\mathbf x)=\prod_{i=1}^n\exp\Big(\langle J_i,A^i\rangle-\frac12J_i^\top MJ_i\Big).
\]
In this product, the factors are independent with respect to the law of the couplings, because the term $i$ only depends on the $i$-th row of the coupling matrix. The averaged density thus conserves this product structure over $i$ of the averages of the factors. We want to compare densities for $J_i=\frac{\bar J}{n}\mathbf 1+\frac{\sigma}{\sqrt n}\xi$, where $\xi$ is either distributed as a row of $\hat J$ (for $\hat Q^n_\sigma$) or $\xi=g$ is a standard Gaussian vector (for $\tilde Q^n_\sigma$). In the density, the terms that do not involve the random part $\xi$ are identical. They thus disappear when considering the ratio of densities, which becomes:
\begin{equation}\label{eq:Bproduct}
\frac{d\hat Q^n_\sigma}{d\tilde Q^n_\sigma}(\mathbf x)=\prod_{i=1}^n\frac{\E\big[\Phi_i(\xi)\big]}{\E\big[\Phi_i(g)\big]},\qquad\text{where}\quad\Phi_i(z):=e^{\langle z,\alpha^i\rangle-\frac12z^\top\Gamma z},\quad z\in\R^n
\end{equation}
with {$\alpha^i:=\frac{\sigma}{\sqrt{n}}\big(A^i-\frac{\bar J}{n}M\mathbf 1\big)$, that is, $\alpha^i_j=\frac{\sigma}{\sqrt n}\big(A^i_j-\frac{\bar J}{n}\sum_{k}M_{jk}\big)$,} and $\Gamma=\frac{\sigma^2}{n}M$. 

In particular, $\hat Q^n_\sigma\ll\tilde Q^n_\sigma$, and to show that they converge to the same limit, it suffices to show that this density is close to $1$, which we quantify through the relative entropy $H(\hat Q^n_\sigma|\tilde Q^n_\sigma)$. In the sequel, we write $\Lambda_Z(\alpha,\Gamma):=\E\exp\big(\langle Z,\alpha\rangle-\frac12Z^\top\Gamma Z\big)$ for a random vector $Z$, so that the $i$-th factor is $\Lambda_\xi(\alpha^i,\Gamma)/\Lambda_g(\alpha^i,\Gamma)$.

\begin{lemma}\label{lem:Brow}
Let $\xi$ have independent coordinates satisfying~\eqref{hyp:sub-Gaussian-tail}, and let $g$ be a standard Gaussian vector. There exist $C,c<\infty$, depending only on $T$, $\sigma$, $\theta_0$ and $D$, such that, for every $n$, every $\alpha\in\R^n$ and every symmetric positive semi-definite matrix $\Gamma$ with $|\Gamma_{jk}|\leq\sigma^2T/n$,
\[
\begin{gathered}\Big|\log\frac{\Lambda_\xi(\alpha,\Gamma)}{\Lambda_g(\alpha,\Gamma)}\Big|\leq C\big(1+\Vert\alpha\Vert^2\big)\qquad\text{and}\\\Big|\frac{\Lambda_\xi(\alpha,\Gamma)}{\Lambda_g(\alpha,\Gamma)}-1\Big|\leq Ce^{c\Vert\alpha\Vert^2}\Big(\sum_{j=1}^n|\alpha_j|^3+n^{-1/2}\Big).\end{gathered}
\]
\end{lemma}
\begin{proof}
Let $Z\in\{\xi,g\}$. We note that the matrix $\Gamma$ is symmetric positive semi-definite, and since  $|S|\leq1$, $|\Gamma_{jk}|\leq\sigma^2T/n$ and $\mathrm{tr}\,\Gamma\leq\sigma^2T$. Since $\Gamma$ is positive semi-definite, $\Lambda_Z\leq\E e^{\langle Z,\alpha\rangle}\leq e^{\kappa\Vert\alpha\Vert^2}$ by the preliminary estimates (1)--(2). By Jensen's inequality, and since the coordinates of $Z$ are independent, centered, with unit variance, $\log\Lambda_Z\geq\E\big[\langle Z,\alpha\rangle-\frac12Z^\top\Gamma Z\big]=-\frac12\mathrm{tr}\,\Gamma\geq-\frac{\sigma^2T}2$. This gives the first bound, and in particular $\Lambda_g\geq e^{-\sigma^2T/2}$.

The second bound will be crucial to show that the ratio of densities goes to $1$. It is proven using a classical Lindeberg replacement method. Let $F(z):=e^{\langle z,\alpha\rangle-\frac12z^\top\Gamma z}$, so that $\Lambda_Z=\E F(Z)$. We replace the coordinates of $\xi$ by those of $g$ one at a time: with $Z^{(j)}:=(g_1,\dots,g_j,\xi_{j+1},\dots,\xi_n)$,
\[
\Lambda_\xi-\Lambda_g=\sum_{j=1}^n\big(\E F(Z^{(j-1)})-\E F(Z^{(j)})\big).
\]
The vectors $Z^{(j-1)}$ and $Z^{(j)}$ only differ in their $j$-th coordinate, $\xi_j$ and $g_j$, which simplifies our estimates because a single modification of a coefficient can be easily estimated as follows. For $W$ an arbitrary vector with $W_j=0$ and for $e_j$ the $j$-th vector of the canonical basis, we have $F(W+se_j)=F(W)\,\psi(s)$, with
\[
\psi(s):=e^{s\beta-\frac12s^2\gamma},\qquad\beta:=\alpha_j-(\Gamma W)_j,\qquad\gamma:=\Gamma_{jj}\in[0,\sigma^2T/n].
\]
Since $\xi_j$ and $g_j$ are independent of $W$ and have the same first two moments, the terms of order at most $2$ of the Taylor expansion of $\psi(\xi_j)$ and $\psi(g_j)$ at $0$ cancel, and
\[
\big|\E F(Z^{(j-1)})-\E F(Z^{(j)})\big|\leq\E\Big[F(W)\big(\E_{\xi_j}|\mathcal R(\xi_j)|+\E_{g_j}|\mathcal R(g_j)|\big)\Big],
\]
where $\mathcal R$ is the Taylor remainder of order $3$ of $\psi$ at $0$. As $\psi'''(u)=\psi(u)\big[(\beta-u\gamma)^3-3\gamma(\beta-u\gamma)\big]$ and $0<\psi(u)\leq e^{|u||\beta|}$, the preliminary estimate (4) and $\gamma\leq\sigma^2T/n$ give $\E|\mathcal R(Z_j)|\leq Ce^{C\beta^2}\big(|\beta|^3+|\beta|/n+n^{-2}\big)$ for $Z_j\in\{\xi_j,g_j\}$. Now $\beta=\alpha_j-Y$, where $Y:=\sum_{k\neq j}\Gamma_{jk}W_k$ is a linear combination of independent sub-Gaussian variables with coefficients at most $\sigma^2T/n$; by the preliminary estimates (2)--(3), for every $\lambda>0$, $\E e^{\lambda Y^2}\leq2$ for $n$ large enough, and $\E|Y|^p\leq C_pn^{-p/2}$. Moreover, $\E F(W)^2\leq\E e^{2\langle W,\alpha\rangle}\leq e^{4\kappa\Vert\alpha\Vert^2}$. By Cauchy--Schwarz, for $n$ large enough, the $j$-th term of the sum is therefore at most $Ce^{c\Vert\alpha\Vert^2}\big(|\alpha_j|^3+|\alpha_j|/n+n^{-3/2}\big)$. Summing over $j$, using $\frac1n\sum_j|\alpha_j|\leq\Vert\alpha\Vert/\sqrt n\leq n^{-1/2}e^{\Vert\alpha\Vert^2}$, and dividing by $\Lambda_g\geq e^{-\sigma^2T/2}$ gives the second bound for $n$ large enough. We conclude by treating the finitely many remaining values of $n$ from the first bound by possibly increasing $C$ and $c$.
\end{proof}

The above bounds involve the second and third powers of $\alpha^i_j$, which are rescaled from $A^i$ with a factor $n^{-1/2}$. To control the terms $A^i$, we now prove the following estimate. 

\begin{lemma}\label{lem:Bmoments}
For every $T>0$, $\sup_n\max_{i,j\leq n}\E_{\hat Q^n_\sigma}\big[{n^2(\alpha^i_j)^4}\big]\leq C<\infty$.
\end{lemma}

\begin{proof}
We work under the coupling conditions discussed above: $\hat Q^n_\sigma$ is the law of the network driven by independent Brownian motions $W^i$, with couplings and initial conditions independent of the $W^i$. From the analysis above, each quenched law $Q(J)$ is equivalent to the measure $P^{\otimes n}$, so $A^i_j$ evaluated at the network paths gives:
\[ \frac{\sqrt n}{\sigma}\alpha^i_j=\int_0^TS(X^j_t)\,dW^i_t+\int_0^TS(X^j_t)\left(\sum_kJ_{ik}S(X^k_t)-\frac{\bar J}{n}\sum_kS(X^k_t)\right)\,dt.\] 
Since $|S|\leq1$, we have $\big|\frac{\bar J}{n}\sum_kS(X^k_t)\big|\leq|\bar J|$, and the Burkholder--Davis--Gundy inequality reduces the claim to $\sup_{n,i}\E\big[(\int_0^T(u^i_t)^2dt)^2\big]<\infty$ with $u^i_t=\sum_kJ_{ik}S(X^k_t)$.

This is the sum of $n$ terms that are not independent: the $i$-th row $(J_{ik})_k$ drives $X^i$, which in turn acts on all the other neurons through the couplings $J_{ki}$, so that the $X^k$ depend on the very weights $J_{ik}$ by which they are multiplied. Conditionally on the trajectories, these weights are neither centered nor independent, which would have been a convenient property for the proof. To exploit the properties of sums of iid random variables, we compare the network with a network where neuron $i$ is removed (a sort of cavity method). In detail, we define a system $(Y^k)_{k\neq i}$ which is a network deprived of neuron $i$, $dY^k_t=\big({f(Y^k_t)}+\sum_{l\neq i}J_{kl}S(Y^l_t)\big)dt+dW^k_t$, $Y^k_0=X^k_0$. We write: 
\[u^i_t=\frac{\bar J}n\sum_kS(X^k_t)+\frac{\sigma}{\sqrt n}\hat J_{ii}S(X^i_t)+U_t+E_t\] 
with
\[
U_t:=\frac{\sigma}{\sqrt n}\sum_{k\neq i}\hat J_{ik}S(Y^k_t),\qquad E_t:=\frac{\sigma}{\sqrt n}\sum_{k\neq i}\hat J_{ik}\big(S(X^k_t)-S(Y^k_t)\big).
\]
The first two terms are bounded by $|\bar J|$ and {$\sigma|\hat J_{ii}|/\sqrt n$}. The process $Y$ is a function of $(J_{kl})_{k,l\neq i}$ and of $(W^k,X^k_0)_{k\neq i}$, hence independent of the $i$-th row; conditionally on $Y$, $U_t$ is a sum of independent centered sub-Gaussian variables with coefficients bounded by $\sigma/\sqrt n$, so $\E U_t^4\leq C$ by the preliminary estimates (2)--(3), and $\E(\int_0^TU_t^2dt)^2\leq T\int_0^T\E U_t^4dt\leq C$. For $E_t$, Cauchy--Schwarz gives $|E_t|\leq\sigma K\big(\frac1n\sum_k\hat J_{ik}^2\big)^{1/2}\Vert D_t\Vert$, with $D^k:=X^k-Y^k$, $k\neq i$. Now
\[
\dot D^k_t={f(X^k_t)-f(Y^k_t)}+\sum_{l\neq i}J_{kl}\big(S(X^l_t)-S(Y^l_t)\big)+J_{ki}S(X^i_t),
\]
and $\Vert(\sum_{l\neq i}J_{kl}(S(X^l_t)-S(Y^l_t)))_k\Vert\leq K\Vert J\Vert_{\mathrm{op}}\Vert D_t\Vert$, $\Vert(J_{ki}S(X^i_t))_k\Vert\leq\Vert J\Vert_{\mathrm{op}}$ (a column norm is at most the operator norm). Since $D_0=0$ {and $\Vert(f(X^k_t)-f(Y^k_t))_k\Vert\leq K\Vert D_t\Vert$}, Gronwall's lemma gives $\Vert D_t\Vert\leq\Vert J\Vert_{\mathrm{op}}\,t\,e^{{K(1+\Vert J\Vert_{\mathrm{op}})t}}$. Therefore
\[
\E\Big(\int_0^TE_t^2dt\Big)^2\leq C\,\E\Big[\Big(\frac1n\sum_k\hat J_{ik}^2\Big)^2\Vert J\Vert_{\mathrm{op}}^4e^{4{K(1+\Vert J\Vert_{\mathrm{op}})}T}\Big],
\]
which is bounded uniformly in $n$ by Cauchy--Schwarz, $\E(\frac1n\sum_k\hat J_{ik}^2)^4\leq\max_k\E\hat J_{ik}^8\leq C$, $\Vert J\Vert_{\mathrm{op}}\leq|\bar J|+\sigma\Vert\hat J\Vert_{\mathrm{op}}/\sqrt n$, and the bound on $\E\exp(\lambda\Vert\hat J\Vert_{\mathrm{op}}/\sqrt n)$ of the preliminary estimates.
\end{proof}

This will allow controlling terms where $\Vert\alpha^i\Vert^2$ is large, but we also need to assure an upper-bound of the probability for that quantity to be small, which requires proving an upper-bound on closed sets and not only on compacts as in~\cite{cabana-touboul:15}. This is however a simple generalization, that we include below for completeness. 
\begin{lemma}\label{lem:Bgauss}
Let $T>0$ and let $F\subset\M_1^+(\C)$ be closed with $Q_\sigma\notin F$. Then $\limsup_n\frac1n\log\tilde Q^n_\sigma(\hat\mu_n\in F)<0$.
\end{lemma}
\begin{proof}
A weak large deviation principle was proven for the system (see e.g.,~\cite[Thms.~4 and~6]{cabana-touboul:15}, applied with a single population and without delays): there exists a good rate function $H_\sigma$ such that $\limsup\frac1n\log\tilde Q^n_\sigma(\hat\mu_n\in K')\leq-\inf_{K'}H_\sigma$ for every compact $K'$, and $H_\sigma$ vanishes only at $Q_\sigma$. We can also show exponential tightness for the sequence. Indeed, for fixed couplings $J$, the density $L_J$ of the quenched law $Q(J)$ with respect to $P^{\otimes n}$ satisfies
\[
L_J^2=\exp\Big(\sum_i\int_0^T2u^i_t\,dW_t(x^i)-\frac12\sum_i\int_0^T(2u^i_t)^2dt\Big)\exp\Big(\int_0^T\Vert J\,S(\mathbf x_t)\Vert^2dt\Big),
\]
where $u^i_t=\sum_jJ_{ij}S(x^j_t)$. By Novikov, the first factor is a $P^{\otimes n}$-martingale with expectation $1$, and $\Vert J\,S(\mathbf x_t)\Vert^2\leq\Vert J\Vert_{\mathrm{op}}^2\,n$. Hence $\E_{P^{\otimes n}}L_J^2\leq e^{T\Vert J\Vert^2_{\mathrm{op}}n}$, and, for any event $E$, $Q(J)(E)\leq e^{T\Vert J\Vert_{\mathrm{op}}^2n/2}P^{\otimes n}(E)^{1/2}$ by Cauchy--Schwarz. {By exponential tightness of the empirical measures of i.i.d.\ variables in a Polish space~\cite[Lemma~6.2.6]{dembo-zeitouni:09}}, for every $M$ there is a compact $K_M\subset\M_1^+(\C)$ with $P^{\otimes n}(\hat\mu_n\notin K_M)\leq e^{-Mn}$ {for all $n$}. Since $\Vert J\Vert_{\mathrm{op}}\leq|\bar J|+\sigma\Vert\tilde J\Vert_{\mathrm{op}}/\sqrt n$, the preliminary estimates imply, with $R_s:=|\bar J|+(2+s)\sigma$,
\[
\tilde Q^n_\sigma(\hat\mu_n\notin K_M)\leq e^{(TR_s^2-M)n/2}+e^{-s^2n/2}.
\]
Given $L>0$, choosing $s$ with $s^2\geq2L$ and then $M\geq TR_s^2+2L$ gives $\limsup\frac1n\log\tilde Q^n_\sigma(\hat\mu_n\notin K_M)\leq-L$. This exponential tightness holds for every $T>0$.

Now that this is established for a compact, let us show our result for any closed set. Let $F$ be closed with $Q_\sigma\notin F$, and let $K_1$ be a compact set as constructed above associated with $L=1$, so that $\tilde Q^n_\sigma(\hat\mu_n\notin K_1)\leq e^{-n/2}$ for $n$ large. The set $F\cap K_1$ is compact and does not contain $Q_\sigma$; since $H_\sigma$ is lower semicontinuous and vanishes only at $Q_\sigma$, its infimum $b$ over $F\cap K_1$ is attained, hence $b>0$ (with $b=+\infty$ if $F\cap K_1=\emptyset$). We thus have $\tilde Q^n_\sigma(\hat\mu_n\in F\cap K_1)\leq e^{-\min(b,1)n/2}$ for $n$ large. Hence
\[
\tilde Q^n_\sigma(\hat\mu_n\in F)\leq\tilde Q^n_\sigma(\hat\mu_n\in F\cap K_1)+\tilde Q^n_\sigma(\hat\mu_n\notin K_1)\leq2e^{-\min(b,1)n/2}
\]
for $n$ large, and $\limsup_n\frac1n\log\tilde Q^n_\sigma(\hat\mu_n\in F)\leq-\frac{\min(b,1)}2<0$.\end{proof}

We can now put all of these estimates together to conclude the proof. 

\begin{proof}[Proof of Theorem~\ref{thm:Convergence4}]
By~\eqref{eq:Bproduct}, $H(\hat Q^n_\sigma|\tilde Q^n_\sigma)\leq\sum_i\E_{\hat Q^n_\sigma}|\log(\Lambda^i_\xi/\Lambda^i_g)|$, where $\Lambda^i_Z:=\Lambda_Z(\alpha^i,\Gamma)$. Fix $i$ and $L\geq1$ and {$V_i:=\sqrt n\sum_j|\alpha^i_j|^3$, so that $\sum_j|\alpha^i_j|^3=n^{-1/2}V_i$}. By the second bound of Lemma~\ref{lem:Brow}, on $\{\Vert\alpha^i\Vert^2 \leq L\}$, the relative error $\varepsilon_i:=|\Lambda^i_\xi/\Lambda^i_g-1|$ is at most $\eta_i:=Ce^{cL}n^{-1/2}(1+V_i)$. We use $|\log(1+x)|\leq2|x|$ on $\{\varepsilon_i\leq\frac12\}$, and the first bound of Lemma~\ref{lem:Brow} elsewhere; since $\mathbf 1_{\{\varepsilon_i>1/2,\,\Vert\alpha^i\Vert^2\leq L\}}\leq2\eta_i$, this gives
\[
\Big|\log\frac{\Lambda^i_\xi}{\Lambda^i_g}\Big|\leq C(1+L)\,\eta_i+C(1+\Vert\alpha^i\Vert^2)\mathbf 1_{\{\Vert\alpha^i\Vert^2>L\}}.
\]
By Lemma~\ref{lem:Bmoments} and Jensen's inequality, $\E_{\hat Q^n_\sigma}V_i\leq C$ and $\E_{\hat Q^n_\sigma}(1+\Vert\alpha^i\Vert^2)^2\leq2+{2n^2\max_j\E_{\hat Q^n_\sigma}(\alpha^i_j)^4}\leq C$. Since $(1+\Vert\alpha^i\Vert^2)\mathbf 1_{\{\Vert\alpha^i\Vert^2>L\}}\leq(1+\Vert\alpha^i\Vert^2)^2/L$,
\[
\E_{\hat Q^n_\sigma}\Big|\log\frac{\Lambda^i_\xi}{\Lambda^i_g}\Big|\leq C\Big((1+L)e^{cL}n^{-1/2}+\frac1L\Big).
\]
With $L:=\max(1,\frac{\log n}{4c})$, $e^{cL}\leq e^cn^{1/4}$, so the right-hand side is at most $C/\log n$, and $H(\hat Q^n_\sigma|\tilde Q^n_\sigma)\leq Cn/\log n$. Let $\epsilon>0$ and $A:=\{\hat\mu_n\notin B(Q_\sigma,\epsilon)\}$. By Lemma~\ref{lem:Bgauss}, there is $b_0>0$ with $\tilde Q^n_\sigma(A)\leq e^{-b_0n}$ for $n$ large. The Donsker--Varadhan variational formula, $\int f\,d\mu\leq H(\mu|\nu)+\log\int e^f\,d\nu$, applied to $f={\ell}\mathbf 1_A$ with ${\ell}=\log(1/\tilde Q^n_\sigma(A))$ {(if $\tilde Q^n_\sigma(A)=0$, then $\hat Q^n_\sigma(A)=0$ since $\hat Q^n_\sigma\ll\tilde Q^n_\sigma$)}, gives ${\ell}\,\hat Q^n_\sigma(A)\leq H(\hat Q^n_\sigma|\tilde Q^n_\sigma)+\log2$, hence $\hat Q^n_\sigma(A)\leq(Cn/\log n+\log2)/(b_0n)\leq C/\log n$ for $n$ large enough; increasing $C$ covers the finitely many remaining values of $n$.
\end{proof}

\section{Proofs of the results on heavy-tailed couplings}\label{app:ProofsHeavy}
\begin{proof}[Proof of Lemma~\ref{lem:LePageTail}]
{The almost sure convergence of $X$ and the fact that its law only depends on the law of $v_1$ are classical~\cite{lepage-woodroofe-zinn:81}, \cite[Chapter~5]{ledoux-talagrand:91}; we only check the explicit tail bound. The points $r_k:=\Gamma_k^{-1/\alpha}\Vert v_k\Vert$ form a Poisson point process on $(0,\infty)$ with mean measure $\nu((u,\infty))=m_\alpha u^{-\alpha}$. Let us now split $X=X_{>u}+X_{\leq u}$ according to whether $r_k>u$ or $r_k\leq u$. Then $\P(X_{>u}\neq0)\leq m_\alpha u^{-\alpha}$, and since the signs are independent and centered, $\Exp[\varepsilon_k\varepsilon_l]=\delta_{kl}$, so that, conditionally on $(\Gamma_k,v_k)_k$, $\Exp\Vert X_{\leq u}\Vert^2=\sum_kr_k^2\mathbf 1_{r_k\leq u}$; taking expectations (Campbell's formula), $\Exp\Vert X_{\leq u}\Vert^2=\int_0^ur^2\,\nu(dr)=\frac{\alpha}{2-\alpha}m_\alpha u^{2-\alpha}$, so that Chebyshev's inequality gives $\P(\Vert X\Vert>u)\leq m_\alpha u^{-\alpha}+\frac{\alpha}{2-\alpha}m_\alpha u^{-\alpha}=\frac{2}{2-\alpha}m_\alpha u^{-\alpha}$. Finally, with $c:=\frac{2}{2-\alpha}m_\alpha$, $\Exp[(\Vert X\Vert\wedge1)^\alpha]=\int_0^1\P(\Vert X\Vert>v^{1/\alpha})\,dv\leq\int_0^1\min(1,c/v)\,dv\leq\psi(c)$.}
\end{proof}

\begin{proof}[Proof of Theorem~\ref{thm:HeavyWellPosed}]
\emph{(i) Existence.} The existence relies on Schauder–Tychonoff's theorem that assures existence of a fixed point for \emph{continuous} function of a non-empty, convex and compact set $\mathcal K$ into itself. Here, the equation of the conjectured limit is a fixed point equation on $\M_1^+(\C)$, $Q=\Phi(Q)$. 

We start by identifying a suitable space $\mathcal K$. By Gronwall's lemma, the solution $x$ of~\eqref{eq:HeavyLimitFixedQ} satisfies, for $s,t\leq T$,
\[
\begin{gathered}\sup_{u\leq T}|x_u|\leq C\Big(1+|x_0|+\Vert Z^Q\Vert_{L^2}+\sup_{u\leq T}|W_u|\Big),\\|x_t-x_s|\leq C\Big(1+\sup_{u\leq T}|x_u|\Big)|t-s|+\Vert Z^Q\Vert_{L^2}|t-s|^{\frac12}+\lambda|W_t-W_s|,\end{gathered}
\]
with $C$ depending only on $T$, $\lambda$, $f$, $\mu$ and $S$. Since the tail of $\Vert Z^Q\Vert_{L^2}$ is bounded uniformly in $Q$ by~\eqref{eq:ZTailUniform}, the family $\{\Phi(Q):Q\in\M_1^+(\C)\}$ is tight, by the Ascoli-Arzel\`a criterion~\cite[Theorem~7.3]{billingsley:99}. Convex combinations and weak limits of a tight family remain tight, so its closed convex hull $\mathcal K$ is a convex compact subset of the space of finite signed measures on $\C$ with the weak topology, and $\Phi(\mathcal K)\subset\mathcal K$. 
%\newtext{Indeed, let $(K_\varepsilon)_{\varepsilon>0}$ be compact sets with $\Phi(Q)(K_\varepsilon^c)\leq\varepsilon$ for all $Q$ and $\varepsilon$. The set of probability measures $\nu$ such that $\nu(K_\varepsilon^c)\leq\varepsilon$ for all $\varepsilon>0$ is convex, closed for the topology of weak convergence (the $K_\varepsilon^c$ are open), hence compact by Prokhorov's theorem. Since the weak topology $\sigma(\mathcal M,C_b)$ of finite signed measures induces on $\M_1^+(\C)$ the topology of weak convergence, this set is also compact, hence closed, among signed measures, and it contains $\mathcal K$. Finally, this topology is metrizable on $\M_1^+(\C)$, so that sequential continuity of $\Phi$ suffices.}

We thus have found a suitable space $\mathcal{K}$ to apply Schauder–Tychonoff's theorem, and now show that $\Phi$ is continuous on $\mathcal{K}$. Let us consider $Q_\ell\to Q$ in $\mathcal K$; we want to show that $\Phi(Q_\ell)\to \Phi(Q)$, and to this purpose it is more convenient to work with almost sure convergence. By Skorokhod's representation theorem, we can take i.i.d.\ pairs $(y^k_\ell,y^k)_k$ with $y^k_\ell\sim Q_\ell$, $y^k\sim Q$ and $\sup_t|y^k_{\ell,t}-y^k_t|\to0$ almost surely. We define $Z^{Q_\ell}$ and $Z^Q$ with the same $(\Gamma_k,\varepsilon_k)$, and $x^\ell$ and $x$ with the same $(x_0,W)$. Lemma~\ref{lem:LePageTail} in $L^2(0,T)$, with $v_k=S(y^k_\ell)-S(y^k)$, and dominated convergence give $\Vert Z^{Q_\ell}-Z^Q\Vert_{L^2}\to0$ in probability; moreover $m_{Q_\ell}\to m_Q$ uniformly. By Gronwall's lemma, $\sup_t|x^\ell_t-x_t|\to0$ in probability, hence $\Phi(Q_\ell)\to\Phi(Q)$. This proves that $\Phi$ is continuous on $\mathcal K$, and the Schauder--Tychonoff theorem ensures the existence of a fixed point.

\emph{(ii) Uniqueness.} Let $Q,Q'$ be two fixed points, and $\Pi$ the set of couplings of $Q$ and $Q'$, a convex compact set. For $\pi\in\Pi$, let $(y^k,y'^k)_k$ be i.i.d.\ with law $\pi$, and let $x$ and $x'$ solve~\eqref{eq:HeavyLimitFixedQ} with inputs $(y^k)$ and $(y'^k)$ respectively, the same $(\Gamma_k,\varepsilon_k)$ and the same $(x_0,W)$. Since $Q$ and $Q'$ are fixed points, the law $\mathcal T\pi$ of $(x,x')$ belongs to $\Pi$. The argument of (i), applied to pairs of paths, shows that $\mathcal T$ is continuous on $\Pi$, so that it has a fixed point $\pi=\mathcal T\pi$. Let $\rho(t):=\Exp_\pi\big[(|y_t-y'_t|\wedge1)^\alpha\big]$.

As $f$ is affine, $\delta x:=x-x'$ solves a linear equation without noise:
\[
\begin{gathered}\delta x_t=D_t+C_\alpha\sum_k\varepsilon_k\Gamma_k^{-1/\alpha}b^k_t,\\ D_t:=\int_0^te^{-\kappa(t-s)}\big(m_Q(s)-m_{Q'}(s)\big)ds,\quad b^k_t:=\int_0^te^{-\kappa(t-s)}\big(S(y^k_s)-S(y'^k_s)\big)ds.\end{gathered}
\]
The integral and the series can be exchanged, since the series $\sum_k\varepsilon_k\Gamma_k^{-1/\alpha}\big(S(y^k)-S(y'^k)\big)$ converges in $L^2(0,T)$ by Lemma~\ref{lem:LePageTail}, and $z\mapsto\int_0^te^{-\kappa(t-s)}z_s\,ds$ is continuous on $L^2(0,T)$. Since $|S(u)-S(v)|\leq C(|u-v|\wedge1)$, $|m_Q(s)-m_{Q'}(s)|\leq|\mu|\,\Exp_\pi|S(y_s)-S(y'_s)|$ and $\alpha\geq1$, Jensen's inequality gives $|D_t|^\alpha+\Exp|b^1_t|^\alpha\leq C\int_0^t\rho(s)\,ds$. Lemma~\ref{lem:LePageTail} with $H=\R$, applied to the series at fixed $t$, together with $(u+v)\wedge1\leq u\wedge1+v\wedge1$ and $u\leq\psi(u)$, and the fact that $\mathcal T\pi=\pi$, yields a constant $K$ depending only on $(\alpha,T,\kappa,\mu,S)$ such that
\begin{equation}\label{eq:OsgoodStep}
\rho(t)\leq K\,\psi\Big(K\int_0^t\rho(s)\,ds\Big),\qquad t\in[0,T].
\end{equation}
Hence $F(t):=K\int_0^t\rho(s)\,ds$ satisfies $F(t)\leq K^2\int_0^t\psi(F(s))\,ds$. Since $\psi$ is continuous and nondecreasing, with $\int_0\frac{du}{\psi(u)}=+\infty$, Osgood's lemma (see, e.g.,~\cite[Chapter~3]{bahouri-chemin-danchin:11}) gives $F\equiv0$, hence $\rho\equiv0$. Under $\pi$, $y_t=y'_t$ almost surely for every $t$, hence $y=y'$ by continuity of the paths, and $Q=Q'$.
\end{proof}

\section{Numerical methods}\label{app:numerics}

All simulations concern the network
\begin{equation}\label{eq:app-network}
dX^i_t=\Big(-X^i_t+I+\sum_{j=1}^n J_{ij}\tanh(X^j_t)\Big)\,dt+\lambda\,dW^i_t,\qquad X^i_0\overset{\text{i.i.d.}}{\sim}\mathcal N(0,1),
\end{equation}
with $\lambda=1$, {$I=1$ for light-tailed and $I=0$ for heavy-tailed couplings (Section~\ref{sec:Numerics}),} discretized by the Euler--Maruyama scheme with time step $\Delta t=10^{-2}$ on $[0,T]$, $T=20$ unless stated otherwise. States are recorded every $10$ steps. The code is written in Python (NumPy, SciPy). We describe here the algorithms and the choices that matter for the comparisons.

\subsection{Coupling distributions}
\emph{Light-tailed couplings (weakly and strongly heterogeneous cases).} We set $J_{ij}=\bar J/n+\sigma g_{ij}/\sqrt{a_n}$, where the $g_{ij}$ are i.i.d., centered, with unit variance with laws chosen among: Gaussian; Rademacher; uniform (that is, $\sqrt{12}(U-\tfrac12)$ where $U$ is uniform in $[0,1]$); centered exponential ($E-1$ where $E$ is an exponential random variable with rate 1); or a centered and normalized log-normal ($(e^{s\xi}-e^{s^2/2})/\sqrt{(e^{s^2}-1)e^{s^2}}$ with $\xi\sim\mathcal N(0,1)$ and $s=1$). The strongly heterogeneous regime corresponds to $a_n=n$; in the weakly heterogeneous regime we use $a_n\in\{n\log n,n^{3/2},n^2\}$.

\emph{Heavy-tailed couplings.} We set $J_{ij}=g_{ij}/n^{1/\alpha}$ with $g_{ij}$ symmetric and normalized so that
\begin{equation}\label{eq:app-tail}
\P(|g_{ij}|>x)\sim x^{-\alpha},\qquad x\to\infty,
\end{equation}
i.e. with tail constant equal to $1$. This normalization is essential: with it, $n^{-1/\alpha}\sum_j g_{ij}V_j$ converges, for i.i.d. bounded $V_j$ independent of the $g_{ij}$, to the LePage series $\sum_k\varepsilon_k\Gamma_k^{-1/\alpha}V_k$ with constant $1$, so that all laws below share the same conjectured limit. We use
\begin{itemize}
\item symmetric Pareto: $g=\varepsilon\,U^{-1/\alpha}$, for which~\eqref{eq:app-tail} holds exactly for $x\ge 1$;
\item symmetric $\alpha$-stable with characteristic function $e^{-|\theta|^\alpha}$, divided by $C_\alpha^{1/\alpha}$ with $C_\alpha=\frac2\pi\Gamma(\alpha)\sin\frac{\pi\alpha}2$;
\item Student $t_\nu$ with $\nu=\alpha$, divided by $c^{1/\alpha}$ with $c=\frac{2}{\nu}\,\frac{\Gamma(\frac{\nu+1}2)}{\sqrt{\nu\pi}\,\Gamma(\frac\nu2)}\,\nu^{\frac{\nu+1}2}$.
\end{itemize}
For finite variance but heavy tails ($\alpha\in(2,4)$), symmetric Pareto weights are normalized to unit variance and used with $J_{ij}=\bar J/n+\sigma g_{ij}/\sqrt n$.

\subsection{Simulations of limit equations}
{
The three limit equations are self-consistent: an input $u$ received by a neuron is built from the law of the solution itself (through $\Exp[S(x_t)]$, through the covariance $\Exp[S(x_t)S(x_s)]$, or through the $\alpha$-stable process $Z^x$). We approximate the law of the solution by the empirical law of a population of $M$ particles, and compute the fixed point by Picard iterations. Each iteration consists of two steps:
\begin{enumerate}
\item \emph{Simulation.} Given the current statistics of the input, each particle $p\le M$ receives an input $u^p$ sampled from these statistics, and is simulated on the grid $t_l=l\Delta t$ by the Euler--Maruyama scheme
\[
x^p_{t_{l+1}}=x^p_{t_l}+\Delta t\,\big(-x^p_{t_l}+I+u^p(t_l)\big)+\lambda\sqrt{\Delta t}\,\xi^p_l,
\]
with initial conditions $x^p_0$ i.i.d.\ $\mathcal N(0,1)$ and i.i.d.\ standard Gaussian increments $\xi^p_l$.
\item \emph{Update.} The statistics of the input are recomputed from the empirical law of the $M$ simulated trajectories.
\end{enumerate}
The initial conditions and the increments $\xi^p_l$ are drawn once and kept fixed across iterations. For light tails, we also use the damped (Krasnosel'skii--Mann~\cite{mann:53}) version of the iteration, in which the new statistics are the average of the current statistics and of the updated ones. 

\emph{McKean--Vlasov limit.} The input is deterministic, $u=m$ with $m(t)=\bar J\,\Exp[\tanh(x_t)]$, and the iterated quantity is the function $m$ on the time grid. Starting from $m\equiv0$, each iteration simulates the particles with input $m$ and computes $\tilde m(t_l)=\bar J\,\frac1M\sum_{p}\tanh(x^p_{t_l})$; the new input is $\frac12(m+\tilde m)$. We use $M=40\,000$ particles, and stop when $\max_l|\tilde m(t_l)-m(t_l)|<10^{-4}$, or after $15$ iterations.

\emph{Gaussian self-consistent limit.} The input is $u=m+G$, where $G$ is a centered Gaussian process with covariance $K(t,s)=\sigma^2\Exp[\tanh(x_t)\tanh(x_s)]$, and the iterated quantity is the pair $(m,K)$ on the time grid. To sample $G$, we compute the Cholesky factor $L$ of the matrix $\big(K(t_l,t_{l'})\big)_{l,l'}+10^{-8}\,\mathrm{Id}$, and set $G^p=Lz^p$, where the $z^p$ are standard Gaussian vectors, drawn once and kept fixed across iterations. Starting from $m\equiv0$ and $K\equiv\sigma^2/2$, each iteration simulates the particles with inputs $m+G^p$, computes the empirical mean $\tilde m$ as above and the empirical covariance $\tilde K(t_l,t_{l'})=\sigma^2\frac1M\sum_p\tanh(x^p_{t_l})\tanh(x^p_{t_{l'}})$, and sets the new statistics to $\frac12\big((m,K)+(\tilde m,\tilde K)\big)$. We use $M=16\,000$ particles, and stop when $\max_l|\tilde m(t_l)-m(t_l)|$ and $\max_{l,l'}|\tilde K(t_l,t_{l'})-K(t_l,t_{l'})|$ are below $2\cdot10^{-3}$, or after $60$ iterations; in practice, the iteration stops after about $45$ steps.

\emph{Conjectured $\alpha$-stable limit.} The input is $u=\mu\,\Exp[\tanh(x_t)]+Z_t$, where $Z$ is the LePage series $Z_t=\sum_{k\geq1}\varepsilon_k\Gamma_k^{-1/\alpha}\tanh(y^k_t)$ of the conjecture, with $(\Gamma_k)$ the arrival times of a unit-rate Poisson process, $(\varepsilon_k)$ independent uniform signs, and $(y^k)$ independent copies of the solution. Here, the iterated object is the population $(y^q)_{q\le M}$ itself, which serves as a sample of the law of the solution. For each particle $p$ and at each iteration, we draw independently:
\begin{itemize}
\item the first $k_0$ arrival times $\Gamma^p_1<\dots<\Gamma^p_{k_0}$ of a unit-rate Poisson process, and signs $\varepsilon^p_k$;
\item indices $q^p_k$ uniform in $\{1,\dots,M\}$, so that $y^{q^p_k}$ plays the role of an independent copy of the solution;
\item a centered Gaussian process $R^p$ replacing the remainder of the series, $\sum_{k>k_0}\varepsilon_k\Gamma_k^{-1/\alpha}\tanh(y^k_t)$. Since the signs are independent and centered, this remainder has covariance $r_{k_0}\,\Exp[\tanh(y_t)\tanh(y_s)]$, with
\[
r_{k_0}=\sum_{k>k_0}\Exp\big[\Gamma_k^{-2/\alpha}\big]=\sum_{k>k_0}\frac{\Gamma(k-2/\alpha)}{\Gamma(k)};
\]
we give $R^p$ the covariance $r_{k_0}\frac1M\sum_q\tanh(y^q_t)\tanh(y^q_s)$ (plus $10^{-8}$ on the diagonal), and compute $r_{k_0}$ exactly up to $k=10^5$, the rest of the sum being replaced by $\int_{10^5}^\infty x^{-2/\alpha}\,dx$.
\end{itemize}
Particle $p$ is then simulated with input
\[
u^p_t=\mu\,\frac1M\sum_{q}\tanh(y^q_t)+\sum_{k=1}^{k_0}\varepsilon^p_k(\Gamma^p_k)^{-1/\alpha}\tanh\big(y^{q^p_k}_t\big)+R^p_t,
\]
and the population $(y^q)$ is replaced by the $M$ new trajectories, without damping. The population is initialized by simulating the particles with input $u\equiv0$. Since the series is resampled at each iteration, the iteration is random, and the fixed point is only reached in law, up to Monte Carlo fluctuations. We monitor convergence through the Kolmogorov--Smirnov distance between successive populations (Figure~\ref{fig:numerics}a), and use $k_0=300$ ($k_0=200$ for Figure~\ref{fig:universality}), $M=8\,000$ particles and $12$ iterations.
}

\begin{figure}
	\includegraphics[width=\textwidth]{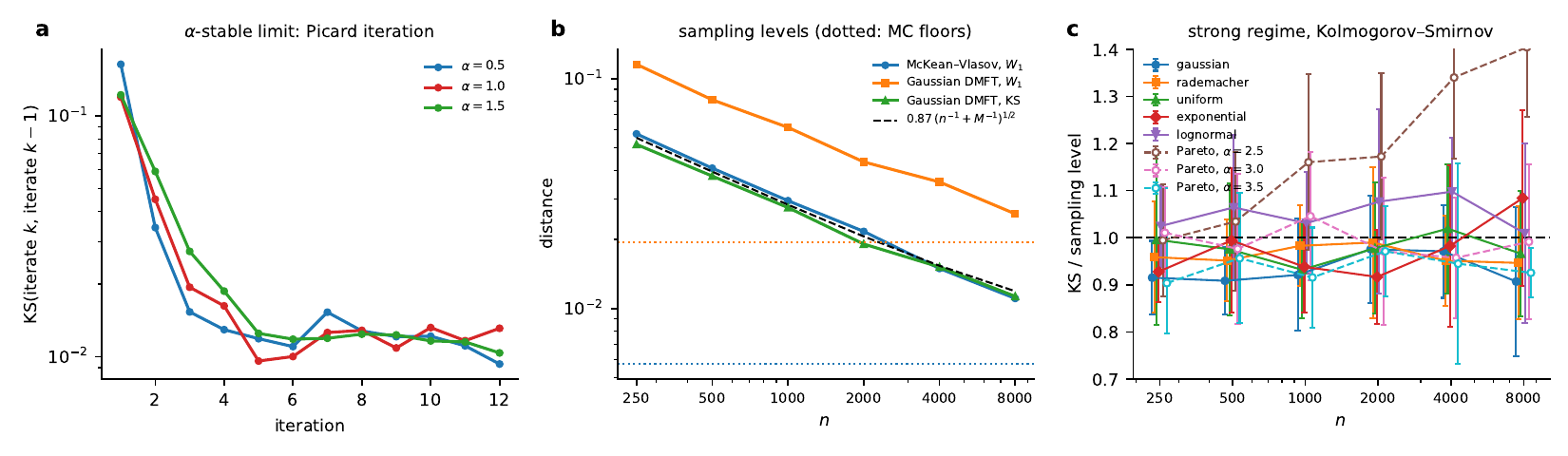}
	\caption{{Numerical controls. (a)~Kolmogorov--Smirnov distance between successive iterates of the fixed-point iteration for the $\alpha$-stable limit. (b)~Sampling levels: distance between $n$ trajectories drawn from an independent solution of the limit and the reference solution, for the McKean--Vlasov and Gaussian limits; dashed: $0.87\,(n^{-1}+M^{-1})^{1/2}$ with $M=16\,000$; dotted: Monte Carlo floors. (c)~Strong regime: Kolmogorov--Smirnov distance between the network and the limit, divided by the sampling level.}}
	\label{fig:numerics}
\end{figure}

\subsection{Comparison with the limit: observables and distances}\label{app:distances}
For a network of size $n$, we compare, at a given time $t$, the empirical law of the $n$ values $(x^1_t,\dots,x^n_t)$ with the empirical law of the $M$ samples of the limit. For light tails we use the Wasserstein distance $W_1$; for heavy tails, whose first moment may be infinite, we use the Kolmogorov--Smirnov distance~\cite{kolmogorov:33,smirnov:39} (the largest difference between the two distribution functions). Distances are averaged over $8$ equally spaced times in $[T/10,T]$, and each point in the figures is the mean and standard deviation over {$10$} independent realizations of the couplings. All laws of the couplings use the same initial conditions and Brownian motions, so that differences between them are due to the couplings only. 

A reference distance was estimated by drawing $n$ trajectories from a second, independent numerical solution of the limit, and measuring their distance to the reference solution with exactly the same procedure as for the network. For the Kolmogorov--Smirnov distance it is also given by an explicit formula $\sqrt{\pi/2}\ln(2)\,(n^{-1}+M^{-1})^{1/2}$. This formula assumes that the $n$ neurons are independent, which only holds asymptotically; its agreement with the numerical estimate (Figure~\ref{fig:numerics}b) is therefore a consistency check. These two estimates give us a natural comparison: if the distance between the network and the limit is on the same order as the sampling level, the network cannot be distinguished from $n$ independent copies of the limit at this size.

For a network trajectory $(X^i_{t_k})$ and a sample $(x^p_{t_k})$ of the limit, we compute:
\begin{itemize}
\item the Wasserstein-$1$ distance $W_1$ between the empirical laws of $(X^i_{t})_{i\le n}$ and $(x^p_t)_{p\le M}$ (light tails);
\item the Kolmogorov--Smirnov distance between the same marginals (heavy tails, for which $W_1$ is dominated by extreme values);
\item the stationary correlation $C(\tau)=\langle\tanh(x_t)\tanh(x_{t+\tau})\rangle$, averaged over neurons and over $t\in[T/2,T]$, and the $L^2$ distance between network and limit;
\item the quenched fluctuations, i.e. the standard deviation of $\frac1n\sum_i\tanh(X^i_T)$ over realizations of $J$;
\item for heavy tails, the per-neuron quantities $\sum_jJ_{ij}^2$ and the integrated correlation time of each $X^i$.
\end{itemize}
Each limit is computed twice with independent Monte Carlo samples; the distance between the two copies defines a \emph{Monte Carlo floor}, below which differences are not significant. Network sizes are $n\in\{250,500,\dots,8000\}$.

\begin{figure}
	\includegraphics[width=\textwidth]{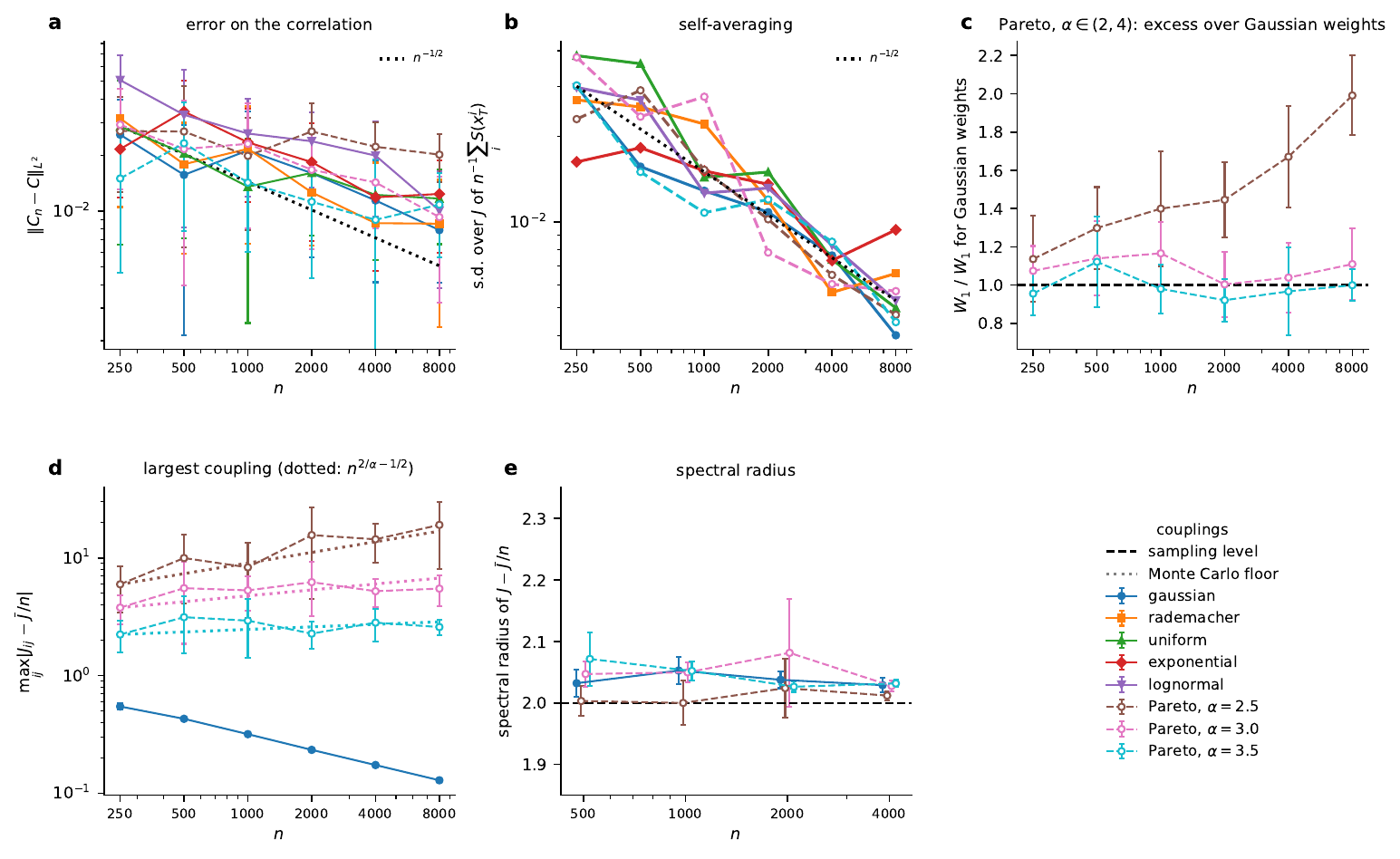}
	\caption{{Strongly heterogeneous regime, further analyses (panels a--d: same simulations as Figure~\ref{fig:lighttails}c--e; panel e: separate simulations with $I=0$).} (a)~$L^2$ error on the correlation function, $\tau\in[0,5]$; (b)~standard deviation, over realizations of the couplings, of the population average $n^{-1}\sum_iS(x^i_T)$; (c)~distance for Pareto couplings divided by the distance for Gaussian couplings; (d)~largest coupling, dotted: $n^{2/\alpha-1/2}$; (e)~spectral radius of $J-\bar J/n$ ($5$ realizations, $n\leq4\,000$), dashed: $\sigma$.}
	\label{fig:strongsupp}
\end{figure}


\begin{thebibliography}{999}

\bibitem{aldous-steele:04}
D.~Aldous and J.~M.~Steele. \newblock The objective method: probabilistic combinatorial optimization and local weak convergence. \newblock In H.~Kesten (ed.), {\em Probability on Discrete Structures}, Encyclopaedia Math. Sci. 110, pages 1--72. Springer, 2004.

\bibitem{aljadeff-stern-sharpee:15}
J.~Aljadeff, M.~Stern and T.~Sharpee. \newblock Transition to chaos in random networks with cell-type-specific connectivity. \newblock {\em Phys. Rev. Lett.}, 114:088101, 2015.

\bibitem{allesina-tang:12}
{S.~Allesina and S.~Tang. \newblock Stability criteria for complex ecosystems. \newblock {\em Nature}, 483(7388):205--208, 2012.}

\bibitem{amari:71}
S.-I.~Amari. \newblock Characteristics of randomly connected threshold-element networks and network systems. \newblock {\em Proceedings of the IEEE}, 59(1):35--47, 1971.

\bibitem{arnese-lacker:26}
M.~Arnese and D.~Lacker. \newblock Sharp propagation of chaos for mean field Langevin dynamics, control, and games. \newblock {\em arXiv:2603.10988}, 2026.

\bibitem{auffinger-benarous-peche:09}
A.~Auffinger, G.~Ben~Arous and S.~P\'ech\'e. \newblock Poisson convergence for the largest eigenvalues of heavy tailed random matrices. \newblock {\em Ann. Inst. H. Poincar\'e Probab. Statist.}, 45(3):589--610, 2009.

\bibitem{ayi-pouradierduteil:21}
{N.~Ayi and N.~Pouradier~Duteil. \newblock Mean-field and graph limits for collective dynamics models with time-varying weights. \newblock {\em J. Differential Equations}, 299:65--110, 2021.}

\bibitem{ayi-pouradierduteil:24}
{N.~Ayi and N.~Pouradier~Duteil. \newblock Large-population limits of non-exchangeable particle systems. \newblock In {\em Active Particles, Volume~4}, pages 79--133. Birkh\"auser, 2024.}

\bibitem{bahouri-chemin-danchin:11}
H.~Bahouri, J.-Y.~Chemin and R.~Danchin. \newblock {\em Fourier Analysis and Nonlinear Partial Differential Equations}. \newblock Grundlehren Math. Wiss. 343, Springer, Heidelberg, 2011.

\bibitem{bai:97}
Z.~D.~Bai. \newblock Circular law. \newblock {\em Ann. Probab.}, 25(1):494--529, 1997.

\bibitem{bandeira-vanhandel:16}
A.~S.~Bandeira and R.~van~Handel. \newblock Sharp nonasymptotic bounds on the norm of random matrices with independent entries. \newblock {\em Ann. Probab.}, 44(4):2479--2506, 2016.

\bibitem{barral-reyes:16}
J.~Barral and A.~D.~Reyes. \newblock Synaptic scaling rule preserves excitatory--inhibitory balance and salient neuronal network dynamics. \newblock {\em Nature Neurosci.}, 19(12):1690--1696, 2016.

\bibitem{bayraktar-chakraborty-wu:23}
E.~Bayraktar, S.~Chakraborty and R.~Wu. \newblock Graphon mean field systems. \newblock {\em Ann. Appl. Probab.}, 33(5):3587--3619, 2023.

\bibitem{ben-arous-dembo-guionnet:01}
G.~Ben~Arous, A.~Dembo and A.~Guionnet. \newblock Aging of spherical spin glasses. \newblock {\em Probab. Theory Related Fields}, 120(1):1--67, 2001.

\bibitem{ben-arous-dembo-guionnet:06}
G.~Ben~Arous, A.~Dembo and A.~Guionnet. \newblock Cugliandolo--Kurchan equations for dynamics of spin-glasses. \newblock {\em Probab. Theory Related Fields}, 136(4):619--660, 2006.

\bibitem{ben-arous-guionnet:95}
G.~Ben~Arous and A.~Guionnet. \newblock Large deviations for Langevin spin glass dynamics. \newblock {\em Probab. Theory Related Fields}, 102(4):455--509, 1995.

\bibitem{ben-arous-guionnet:97}
G.~Ben~Arous and A.~Guionnet. \newblock Symmetric Langevin spin glass dynamics. \newblock {\em Ann. Probab.}, 25(3):1367--1422, 1997.

\bibitem{ben-arous-guionnet:08}
G.~Ben~Arous and A.~Guionnet. \newblock The spectrum of heavy tailed random matrices. \newblock {\em Comm. Math. Phys.}, 278:715--751, 2008.

\bibitem{ben-arous-zeitouni:99}
G.~Ben~Arous and O.~Zeitouni. \newblock Increasing propagation of chaos for mean field models. \newblock {\em Ann. Inst. H. Poincar\'e Probab. Statist.}, 35(1):85--102, 1999.

\bibitem{bhamidi-budhiraja-wu:19}
S.~Bhamidi, A.~Budhiraja and R.~Wu. \newblock Weakly interacting particle systems on inhomogeneous random graphs. \newblock {\em Stochastic Process. Appl.}, 129(6):2174--2206, 2019.

\bibitem{billingsley:99}
P.~Billingsley. \newblock {\em Convergence of Probability Measures}. \newblock Wiley, New York, 2nd edition, 1999.

\bibitem{blaustein:22}
{A.~Blaustein. \newblock Large coupling in a FitzHugh--Nagumo neural network: quantitative and strong convergence results. \newblock {\em J. Differential Equations}, 374:218--266, 2023.}

\bibitem{blaustein-bouin:22}
{A.~Blaustein and E.~Bouin. \newblock Concentration profiles in FitzHugh--Nagumo neural networks: a Hopf--Cole approach. \newblock {\em Discrete Contin. Dyn. Syst. Ser. B}, 29(4):2018--2042, 2024.}

\bibitem{blaustein-filbet:22}
{A.~Blaustein and F.~Filbet. \newblock Concentration phenomena in FitzHugh--Nagumo equations: a mesoscopic approach. \newblock {\em SIAM J. Math. Anal.}, 55(1):367--404, 2023.}

\bibitem{bordenave-caputo:14}
C.~Bordenave and P.~Caputo. \newblock A large deviation principle for {W}igner matrices without {G}aussian tails. \newblock {\em Ann. Probab.}, 42(6):2454--2496, 2014.

\bibitem{bordenave-caputo-chafai:11}
C.~Bordenave, P.~Caputo and D.~Chafa\"{\i}. \newblock Spectrum of non-Hermitian heavy tailed random matrices. \newblock {\em Comm. Math. Phys.}, 307(2):513--560, 2011.

\bibitem{bordenave-caputo-chafai-tikhomirov:18}
C.~Bordenave, P.~Caputo, D.~Chafa\"{\i} and K.~Tikhomirov. \newblock On the spectral radius of a random matrix: an upper bound without fourth moment. \newblock {\em Ann. Probab.}, 46(4):2268--2286, 2018.

\bibitem{bordenave-guionnet:13}
{C.~Bordenave and A.~Guionnet. \newblock Localization and delocalization of eigenvectors for heavy-tailed random matrices. \newblock {\em Probab. Theory Related Fields}, 157:885--953, 2013.}

\bibitem{brunel:00}
N.~Brunel. \newblock Dynamics of sparsely connected networks of excitatory and inhibitory spiking neurons. \newblock {\em J. Comput. Neurosci.}, 8(3):183--208, 2000.

\bibitem{budhiraja-dupuis-fischer:12}
A.~Budhiraja, P.~Dupuis and M.~Fischer. \newblock Large deviation properties of weakly interacting processes via weak convergence methods. \newblock {\em Ann. Probab.}, 40(1):74--102, 2012.

\bibitem{bunin:17}
G.~Bunin. \newblock Ecological communities with Lotka--Volterra dynamics. \newblock {\em Phys. Rev. E}, 95:042414, 2017.

\bibitem{buzsaki-mizuseki:14}
{G.~Buzs\'aki and K.~Mizuseki. \newblock The log-dynamic brain: how skewed distributions affect network operations. \newblock {\em Nature Rev. Neurosci.}, 15:264--278, 2014.}

\bibitem{cabana2016large}
T.~Cabana. \newblock {\em Large deviations for the dynamics of heterogeneous neural networks}. \newblock PhD dissertation, Universit\'e Pierre et Marie Curie--Paris~VI, 2016.

\bibitem{cabana-touboul:13}
T.~Cabana and J.~Touboul. \newblock Large deviations, dynamics and phase transitions in large stochastic and disordered neural networks. \newblock {\em J. Stat. Phys.}, 153(2):211--269, 2013.

\bibitem{cabana-touboul:15}
T.~Cabana and J.~Touboul. \newblock Large deviations for randomly connected neural networks: I. Spatially extended systems. \newblock {\em Adv. Appl. Probab.}, 50(3):944--982, 2018.

\bibitem{cabana-touboul:16}
T.~Cabana and J.~Touboul. \newblock Large deviations for randomly connected neural networks: II. State-dependent interactions. \newblock {\em Adv. Appl. Probab.}, 50(3):983--1004, 2018.

\bibitem{carmona-hu:06}
P.~Carmona and Y.~Hu. \newblock Universality in Sherrington--Kirkpatrick's spin glass model. \newblock {\em Ann. Inst. H. Poincar\'e Probab. Statist.}, 42(2):215--222, 2006.

\bibitem{castellano-fortunato-loreto:09}
C.~Castellano, S.~Fortunato and V.~Loreto. \newblock Statistical physics of social dynamics. \newblock {\em Rev. Mod. Phys.}, 81(2):591--646, 2009.

\bibitem{celentano-cheng-montanari:21}
M.~Celentano, C.~Cheng and A.~Montanari. \newblock The high-dimensional asymptotics of first order methods with random data. \newblock {\em arXiv:2112.07572}, 2021.

\bibitem{cessac:10}
B.~Cessac. \newblock A view of neural networks as dynamical systems. \newblock {\em Internat. J. Bifur. Chaos}, 20(6):1585--1629, 2010.

\bibitem{chaintron-diez:22}
L.-P.~Chaintron and A.~Diez. \newblock Propagation of chaos: a review of models, methods and applications. I. Models and methods. \newblock {\em Kinet. Relat. Models}, 15(6):895--1015, 2022.

\bibitem{chatterjee:06}
S.~Chatterjee. \newblock A generalization of the Lindeberg principle. \newblock {\em Ann. Probab.}, 34(6):2061--2076, 2006.

\bibitem{chen-kim-sen:23}
{W.-K.~Chen, H.~Kim and A.~Sen. \newblock Some rigorous results on the L\'evy spin glass model. \newblock {\em arXiv:2303.06084}, 2023.}

\bibitem{bouchaud-cizeau:94}
{P.~Cizeau and J.-P.~Bouchaud. \newblock Theory of L\'evy matrices. \newblock {\em Phys. Rev. E}, 50(3):1810--1822, 1994.}

\bibitem{coppini-dietert-giacomin:20}
F.~Coppini, H.~Dietert and G.~Giacomin. \newblock A law of large numbers and large deviations for interacting diffusions on Erd\H{o}s--R\'enyi graphs. \newblock {\em Stoch. Dyn.}, 20(2):2050010, 2020.

\bibitem{crevat-faye-filbet:19}
J.~Crevat, G.~Faye and F.~Filbet. \newblock Rigorous derivation of the nonlocal reaction--diffusion FitzHugh--Nagumo system. \newblock {\em SIAM J. Math. Anal.}, 51(1):346--373, 2019.

\bibitem{crisanti-sompolinsky:87}
A.~Crisanti and H.~Sompolinsky. \newblock Dynamics of spin systems with randomly asymmetric bonds: Langevin dynamics and a spherical model. \newblock {\em Phys. Rev. A}, 36:4922--4939, 1987.

\bibitem{crisanti-sompolinsky:18}
A.~Crisanti and H.~Sompolinsky. \newblock Path integral approach to random neural networks. \newblock {\em Phys. Rev. E}, 98(6):062120, 2018.

\bibitem{cugliandolo-kurchan:93}
L.~F.~Cugliandolo and J.~Kurchan. \newblock Analytical solution of the off-equilibrium dynamics of a long-range spin-glass model. \newblock {\em Phys. Rev. Lett.}, 71:173--176, 1993.

\bibitem{daipra-denhollander:96}
P.~Dai~Pra and F.~den~Hollander. \newblock McKean--Vlasov limit for interacting random processes in random media. \newblock {\em J. Stat. Phys.}, 84:735--772, 1996.

\bibitem{davidson-szarek:01}
K.~R.~Davidson and S.~J.~Szarek. \newblock Local operator theory, random matrices and Banach spaces. \newblock In {\em Handbook of the Geometry of Banach Spaces}, Vol.~I, pages 317--366. North-Holland, Amsterdam, 2001.

\bibitem{dawson-gartner:87}
D.~A.~Dawson and J.~G\"artner. \newblock Large deviations from the McKean--Vlasov limit for weakly interacting diffusions. \newblock {\em Stochastics}, 20(4):247--308, 1987.

\bibitem{delattre-fournier-hoffmann:16}
S.~Delattre, N.~Fournier and M.~Hoffmann. \newblock Hawkes processes on large networks. \newblock {\em Ann. Appl. Probab.}, 26(1):216--261, 2016.

\bibitem{delattre-giacomin-lucon:16}
S.~Delattre, G.~Giacomin and E.~Lu\c{c}on. \newblock A note on dynamical models on random graphs and Fokker--Planck equations. \newblock {\em J. Stat. Phys.}, 165:785--798, 2016.

\bibitem{dembo-gheissari:21}
A.~Dembo and R.~Gheissari. \newblock Diffusions interacting through a random matrix: universality via stochastic Taylor expansion. \newblock {\em Probab. Theory Related Fields}, 180(3--4):1057--1097, 2021.

\bibitem{dembo-lubetzky-zeitouni:21}
A.~Dembo, E.~Lubetzky and O.~Zeitouni. \newblock Universality for Langevin-like spin glass dynamics. \newblock {\em Ann. Appl. Probab.}, 31(6):2864--2880, 2021.

\bibitem{dembo-zeitouni:09}
A.~Dembo and O.~Zeitouni. \newblock {\em Large Deviations Techniques and Applications}. \newblock Stochastic Modelling and Applied Probability 38, Springer, 2nd edition (corrected printing), 2010.

\bibitem{dudeja-lu-sen:23}
{R.~Dudeja, Y.~M.~Lu and S.~Sen. \newblock Universality of approximate message passing with semirandom matrices. \newblock {\em Ann. Probab.}, 51(5), 2023. doi:10.1214/23-AOP1628.}

\bibitem{faugeras-tanre:25}
O.~Faugeras and E.~Tanr\'e. \newblock Universality of the mean-field equations of networks of Hopfield-like neurons. \newblock {\em arXiv:2408.14290}, 2025.

\bibitem{ganguly-ramanan:24}
A.~Ganguly and K.~Ramanan. \newblock Hydrodynamic limits of non-Markovian interacting particle systems on sparse graphs. \newblock {\em Electron. J. Probab.}, 29:1--63, 2024.

\bibitem{geman:82}
S.~Geman. \newblock Almost sure stable oscillations in a large system of randomly coupled equations. \newblock {\em SIAM J. Appl. Math.}, 42(4):695--703, 1982.

\bibitem{geman:1982b}
S.~Geman and C.-R.~Hwang. \newblock A chaos hypothesis for some large systems of random equations. \newblock {\em Z. Wahrscheinlichkeitstheorie verw. Gebiete}, 60(3):291--314, 1982.

\bibitem{gerbelot-etal:24}
{C.~Gerbelot, E.~Troiani, F.~Mignacco, F.~Krzakala and L.~Zdeborov\'a. \newblock Rigorous dynamical mean-field theory for stochastic gradient descent methods. \newblock {\em SIAM J. Math. Data Sci.}, 2024. doi:10.1137/23M1594388.}

\bibitem{ginibre:65}
J.~Ginibre. \newblock Statistical ensembles of complex, quaternion, and real matrices. \newblock {\em J. Math. Phys.}, 6:440--449, 1965.

\bibitem{gnedenko-kolmogorov:54}
B.~V.~Gnedenko and A.~N.~Kolmogorov. \newblock {\em Limit Distributions for Sums of Independent Random Variables}. \newblock Addison-Wesley, 1954.

\bibitem{grunwald:96}
M.~Grunwald. \newblock Sanov results for Glauber spin-glass dynamics. \newblock {\em Probab. Theory Related Fields}, 106(2):187--232, 1996.

\bibitem{guionnet:97}
A.~Guionnet. \newblock Averaged and quenched propagation of chaos for spin glass dynamics. \newblock {\em Probab. Theory Related Fields}, 109:183--215, 1997.

\bibitem{hanin:23}
B.~Hanin. \newblock Random neural networks in the infinite width limit as Gaussian processes. \newblock {\em Ann. Appl. Probab.}, 33(6A):4798--4819, 2023.

\bibitem{helias-dahmen:20}
M.~Helias and D.~Dahmen. \newblock {\em Statistical Field Theory for Neural Networks}. \newblock Lecture Notes in Physics 970, Springer, 2020.

\bibitem{hermann-touboul:12}
G.~Hermann and J.~Touboul. \newblock Heterogeneous connections induce oscillations in large-scale networks. \newblock {\em Phys. Rev. Lett.}, 109:018702, 2012.

\bibitem{jabin-poyato-soler:25}
P.-E.~Jabin, D.~Poyato and J.~Soler. \newblock Mean-field limit of non-exchangeable systems. \newblock {\em Comm. Pure Appl. Math.}, 78(4):651--741, 2025.

\bibitem{jabin-wang:17}
{P.-E.~Jabin and Z.~Wang. \newblock Mean field limit for stochastic particle systems. \newblock In {\em Active Particles, Vol.~1}, pages 379--402. Birkh\"auser, 2017.}

\bibitem{jung-lee-lee-yang:23}
{P.~Jung, H.~Lee, J.~Lee and H.~Yang. \newblock $\alpha$-Stable convergence of heavy-/light-tailed infinitely wide neural networks. \newblock {\em Adv. Appl. Probab.}, 55(4):1415--1441, 2023.}

\bibitem{kac:56}
M.~Kac. \newblock Foundations of kinetic theory. \newblock In {\em Proc. Third Berkeley Symp. Math. Statist. Probab.}, vol.~3, pages 171--197. University of California Press, 1956.

\bibitem{kadmon-sompolinsky:15}
J.~Kadmon and H.~Sompolinsky. \newblock Transition to chaos in random neuronal networks. \newblock {\em Phys. Rev. X}, 5(4):041030, 2015.

\bibitem{kolmogorov:33}
A.~N.~Kolmogorov. \newblock Sulla determinazione empirica di una legge di distribuzione. \newblock {\em Giornale dell'Istituto Italiano degli Attuari}, 4:83--91, 1933.

\bibitem{kuehn-xu:22}
C.~Kuehn and C.~Xu. \newblock Vlasov equations on digraph measures. \newblock {\em J. Differential Equations}, 339:261--349, 2022.

\bibitem{kuramoto-arakai}
Y.~Kuramoto. \newblock Self-entrainment of a population of coupled non-linear oscillators. \newblock In H.~Araki (ed.), {\em International Symposium on Mathematical Problems in Theoretical Physics}, Lecture Notes in Physics 39, page 420. Springer, 1975.

\bibitem{kusmierz-ogawa-toyoizumi:20}
{\leavevmode\setbox0\hbox{L}\hbox to\wd0{\hss\char32L}}.~Ku\'smierz, S.~Ogawa and T.~Toyoizumi. \newblock Edge of chaos and avalanches in neural networks with heavy-tailed synaptic weight distribution. \newblock {\em Phys. Rev. Lett.}, 125:028101, 2020.

\bibitem{lacker:23}
D.~Lacker. \newblock Hierarchies, entropy, and quantitative propagation of chaos for mean field diffusions. \newblock {\em Probab. Math. Phys.}, 4(2):377--432, 2023.

\bibitem{lacker-leflem:23}
D.~Lacker and L.~Le~Flem. \newblock Sharp uniform-in-time propagation of chaos. \newblock {\em Probab. Theory Related Fields}, 187(1--2):443--480, 2023.

\bibitem{lacker-ramanan-wu:21}
D.~Lacker, K.~Ramanan and R.~Wu. \newblock Locally interacting diffusions as Markov random fields on path space. \newblock {\em Stochastic Process. Appl.}, 140:81--114, 2021.

\bibitem{lacker-ramanan-wu:23}
D.~Lacker, K.~Ramanan and R.~Wu. \newblock Local weak convergence for sparse networks of interacting processes. \newblock {\em Ann. Appl. Probab.}, 33(2):643--688, 2023.

\bibitem{lacker-ramanan-wu:23b}
D.~Lacker, K.~Ramanan and R.~Wu. \newblock Marginal dynamics of interacting diffusions on unimodular Galton--Watson trees. \newblock {\em Probab. Theory Related Fields}, 187:817--884, 2023.

\bibitem{lacker-yeung-zhou:24}
D.~Lacker, L.~C.~Yeung and F.~Zhou. \newblock Quantitative propagation of chaos for non-exchangeable diffusions via first-passage percolation. \newblock {\em arXiv:2409.08882}, 2024.

\bibitem{latala:05}
R.~Lata{\l}a. \newblock Some estimates of norms of random matrices. \newblock {\em Proc. Amer. Math. Soc.}, 133(5):1273--1282, 2005.

\bibitem{ledoux-talagrand:91}
M.~Ledoux and M.~Talagrand. \newblock {\em Probability in Banach Spaces. Isoperimetry and Processes}. \newblock Springer, Berlin, 1991.

\bibitem{lepage-woodroofe-zinn:81}
R.~LePage, M.~Woodroofe and J.~Zinn. \newblock Convergence to a stable distribution via order statistics. \newblock {\em Ann. Probab.}, 9(4):624--632, 1981.

\bibitem{lovasz:12}
L.~Lov\'asz. \newblock {\em Large Networks and Graph Limits}. \newblock AMS Colloquium Publications 60, 2012.

\bibitem{lucon:20}
E.~Lu\c{c}on. \newblock Quenched asymptotics for interacting diffusions on inhomogeneous random graphs. \newblock {\em Stochastic Process. Appl.}, 130(11):6783--6842, 2020.

\bibitem{lucon-stannat:14}
E.~Lu\c{c}on and W.~Stannat. \newblock Mean field limit for disordered diffusions with singular interactions. \newblock {\em Ann. Appl. Probab.}, 24(5):1946--1993, 2014.

\bibitem{lynch1987large}
J.~Lynch and J.~Sethuraman. \newblock Large deviations for processes with independent increments. \newblock {\em Ann. Probab.}, 15(2):610--627, 1987.

\bibitem{mann:53}
W.~R.~Mann. \newblock Mean value methods in iteration. \newblock {\em Proc. Amer. Math. Soc.}, 4:506--510, 1953.

\bibitem{martin-siggia-rose:73}
{P.~C.~Martin, E.~D.~Siggia and H.~A.~Rose. \newblock Statistical dynamics of classical systems. \newblock {\em Phys. Rev. A}, 8:423--437, 1973.}

\bibitem{may:72}
R.~M.~May. \newblock Will a large complex system be stable? \newblock {\em Nature}, 238:413--414, 1972.

\bibitem{mckean:66}
H.~P.~McKean. \newblock A class of Markov processes associated with nonlinear parabolic equations. \newblock {\em Proc. Natl. Acad. Sci. USA}, 56:1907--1911, 1966.

\bibitem{medvedev:14}
G.~S.~Medvedev. \newblock The nonlinear heat equation on dense graphs and graph limits. \newblock {\em SIAM J. Math. Anal.}, 46(4):2743--2766, 2014.

\bibitem{mei-montanari-nguyen:18}
{S.~Mei, A.~Montanari and P.-M.~Nguyen. \newblock A mean field view of the landscape of two-layer neural networks. \newblock {\em Proc. Natl. Acad. Sci. USA}, 115(33):E7665--E7671, 2018.}

\bibitem{mezard1987spin}
M.~M\'ezard, G.~Parisi and M.~A.~Virasoro. \newblock {\em Spin Glass Theory and Beyond: An Introduction to the Replica Method and Its Applications}. \newblock World Scientific Lecture Notes in Physics 9, World Scientific, 1987.

\bibitem{mischler-quininao-touboul:16}
S.~Mischler, C.~Qui\~ninao and J.~Touboul. \newblock On a kinetic FitzHugh--Nagumo model of neuronal network. \newblock {\em Comm. Math. Phys.}, 342(3):1001--1042, 2016.

\bibitem{mourier:53}
E.~Mourier. \newblock \'El\'ements al\'eatoires dans un espace de Banach. \newblock {\em Ann. Inst. H. Poincar\'e}, 13(3):161--244, 1953.

\bibitem{moynot-samuelides:02}
O.~Moynot and M.~Samuelides. \newblock Large deviations and mean-field theory for asymmetric random recurrent neural networks. \newblock {\em Probab. Theory Related Fields}, 123(1):41--75, 2002.

\bibitem{oliveira-reis:19}
R.~I.~Oliveira and G.~H.~Reis. \newblock Interacting diffusions on random graphs with diverging degrees: hydrodynamics and large deviations. \newblock {\em J. Stat. Phys.}, 176:1057--1087, 2019.

\bibitem{ostojic:14}
S.~Ostojic. \newblock Two types of asynchronous activity in networks of excitatory and inhibitory spiking neurons. \newblock {\em Nature Neurosci.}, 17:594--600, 2014.

\bibitem{paul-trelat:24}
T.~Paul and E.~Tr\'elat. \newblock Mean field, hydrodynamic and graph limits for deterministic interacting particle systems: a survey with quantitative estimates. \newblock Preprint, arXiv:2209.08832, 2026.

\bibitem{favaro-fortini-peluchetti:20}
S.~Peluchetti, S.~Favaro and S.~Fortini. \newblock Stable behaviour of infinitely wide deep neural networks. \newblock In {\em Proc. 23rd Int. Conf. Artificial Intelligence and Statistics (AISTATS)}, PMLR 108:1137--1146, 2020.

\bibitem{pukhalskii}
A.~A.~Pukhalskii. \newblock On functional principle of large deviations. \newblock In {\em New Trends in Probability and Statistics}, pages 198--218. VSP/Mokslas, 1991.

\bibitem{quininao-touboul:15}
C.~Qui\~ninao and J.~Touboul. \newblock Limits and dynamics of randomly connected neuronal networks. \newblock {\em Acta Appl. Math.}, 136(1):167--192, 2015.

\bibitem{quininao-touboul:20}
C.~Qui\~ninao and J.~Touboul. \newblock Clamping and synchronization in the strongly coupled FitzHugh--Nagumo model. \newblock {\em SIAM J. Appl. Dyn. Syst.}, 19(2), 2020. doi:10.1137/19M1283884.

\bibitem{quininao-touboul:25}
{C.~Qui\~ninao and J.~Touboul. \newblock Balanced dynamics in strongly coupled networks. \newblock {\em SIAM J. Appl. Dyn. Syst.}, 25(2), 2026. doi:10.1137/25M173140X.}

\bibitem{rajan-abbott:06}
K.~Rajan and L.~F.~Abbott. \newblock Eigenvalue spectra of random matrices for neural networks. \newblock {\em Phys. Rev. Lett.}, 97:188104, 2006.

\bibitem{ramanan:23}
K.~Ramanan. \newblock Interacting processes on sparse graphs. \newblock In {\em Proceedings of the International Congress of Mathematicians 2022}, Vol.~6, pages 4394--4425. EMS Press, 2023.

\bibitem{renart-etal:10}
{A.~Renart, J.~de~la~Rocha, P.~Bartho, L.~Hollender, N.~Parga, A.~Reyes and K.~D.~Harris. \newblock The asynchronous state in cortical circuits. \newblock {\em Science}, 327(5965):587--590, 2010.}

\bibitem{rosenbaum-doiron:14}
{R.~Rosenbaum and B.~Doiron. \newblock Balanced networks of spiking neurons with spatially dependent recurrent connections. \newblock {\em Phys. Rev. X}, 4:021039, 2014.}

\bibitem{samorodnitsky-taqqu:94}
G.~Samorodnitsky and M.~S.~Taqqu. \newblock {\em Stable Non-Gaussian Random Processes: Stochastic Models with Infinite Variance}. \newblock Chapman \& Hall, 1994.

\bibitem{cessac-samuelides:07b}
M.~Samuelides and B.~Cessac. \newblock Random recurrent neural networks dynamics. \newblock {\em Eur. Phys. J. Special Topics}, 142(1):89--122, 2007.

\bibitem{sherrington-kirkpatrick:75}
D.~Sherrington and S.~Kirkpatrick. \newblock Solvable model of a spin-glass. \newblock {\em Phys. Rev. Lett.}, 35(26):1792--1796, 1975.

\bibitem{shkolnikov-yeung:26}
M.~Shkolnikov and L.~C.~Yeung. \newblock Universal central limit theorem for non-exchangeable interacting diffusions. \newblock {\em arXiv:2607.07598}, 2026.

\bibitem{sirignano-spiliopoulos:20}
J.~Sirignano and K.~Spiliopoulos. \newblock Mean field analysis of neural networks: a law of large numbers. \newblock {\em SIAM J. Appl. Math.}, 80(2):725--752, 2020.

\bibitem{smirnov:39}
N.~V.~Smirnov. \newblock On the estimation of the discrepancy between empirical curves of distribution for two independent samples. \newblock {\em Bulletin Math\'ematique de l'Universit\'e de Moscou}, 2(2):3--14, 1939.

\bibitem{sompolinsky-crisanti-etal:88}
H.~Sompolinsky, A.~Crisanti and H.~J.~Sommers. \newblock Chaos in random neural networks. \newblock {\em Phys. Rev. Lett.}, 61(3):259--262, 1988.

\bibitem{sompolinsky-zippelius:82}
H.~Sompolinsky and A.~Zippelius. \newblock Relaxational dynamics of the Edwards--Anderson model and the mean-field theory of spin-glasses. \newblock {\em Phys. Rev. B}, 25(11):6860--6875, 1982.

\bibitem{song-etal:05}
S.~Song, P.~J.~Sj\"ostr\"om, M.~Reigl, S.~Nelson and D.~B.~Chklovskii. \newblock Highly nonrandom features of synaptic connectivity in local cortical circuits. \newblock {\em PLoS Biol.}, 3(3):e68, 2005.

\bibitem{soshnikov:04}
A.~Soshnikov. \newblock Poisson statistics for the largest eigenvalues of Wigner random matrices with heavy tails. \newblock {\em Electron. Commun. Probab.}, 9:82--91, 2004.

\bibitem{stern-sompolinsky-abbott:14}
M.~Stern, H.~Sompolinsky and L.~F.~Abbott. \newblock Dynamics of random neural networks with bistable units. \newblock {\em Phys. Rev. E}, 90(6):062710, 2014.

\bibitem{strogatz2000kuramoto}
S.~H.~Strogatz. \newblock From Kuramoto to Crawford: exploring the onset of synchronization in populations of coupled oscillators. \newblock {\em Physica D}, 143:1--20, 2000.

\bibitem{strogatzSync}
S.~H.~Strogatz. \newblock {\em Sync: The Emerging Science of Spontaneous Order}. \newblock Hyperion, 2003.

\bibitem{sznitman:89}
A.-S.~Sznitman. \newblock Topics in propagation of chaos. \newblock In {\em \'Ecole d'\'Et\'e de Probabilit\'es de Saint-Flour XIX---1989}, Lecture Notes in Math. 1464, pages 165--251. Springer, 1991.

\bibitem{talagrand:11}
M.~Talagrand. \newblock {\em Mean Field Models for Spin Glasses}, Vols.~I--II. \newblock Springer, 2011.

\bibitem{tao-vu:10}
T.~Tao and V.~Vu. \newblock Random matrices: universality of ESDs and the circular law. \newblock {\em Ann. Probab.}, 38(5):2023--2065, 2010.

\bibitem{touboul-multiscale}
J.~Touboul. \newblock Spatially extended networks with singular multi-scale connectivity patterns. \newblock {\em J. Stat. Phys.}, 156(3):546--573, 2014.

\bibitem{touboul:14}
J.~Touboul. \newblock Propagation of chaos in neural fields. \newblock {\em Ann. Appl. Probab.}, 24(3):1298--1328, 2014.

\bibitem{vanmeegen-kuhn-helias:21}
A.~van~Meegen, T.~K\"uhn and M.~Helias. \newblock Large-deviation approach to random recurrent neuronal networks: parameter inference and fluctuation-induced transitions. \newblock {\em Phys. Rev. Lett.}, 127(15):158302, 2021.

\bibitem{vanvreeswijk-sompolinsky:96}
C.~van~Vreeswijk and H.~Sompolinsky. \newblock Chaos in neuronal networks with balanced excitatory and inhibitory activity. \newblock {\em Science}, 274:1724--1726, 1996.

\bibitem{vanvreeswijk-sompolinsky:98}
C.~van~Vreeswijk and H.~Sompolinsky. \newblock Chaotic balanced state in a model of cortical circuits. \newblock {\em Neural Comput.}, 10:1321--1371, 1998.

\bibitem{vershynin:18}
R.~Vershynin. \newblock {\em High-Dimensional Probability: An Introduction with Applications in Data Science}. \newblock Cambridge University Press, 2018.

\bibitem{villani:09}
C.~Villani. \newblock {\em Optimal Transport: Old and New}. \newblock Grundlehren Math. Wiss. 338, Springer, 2009.

\bibitem{wardak-gong:21}
A.~Wardak and P.~Gong. \newblock Fractional diffusion theory of balanced heterogeneous neural networks. \newblock {\em Phys. Rev. Research}, 3(1):013083, 2021.

\bibitem{wardak-gong:22}
A.~Wardak and P.~Gong. \newblock Extended Anderson criticality in heavy-tailed neural networks. \newblock {\em Phys. Rev. Lett.}, 129:048103, 2022.

\bibitem{xie-mihalas-kusmierz:25}
E.~Xie, S.~Mihalas and {\leavevmode\setbox0\hbox{L}\hbox to\wd0{\hss\char32L}}.~Ku\'smierz. \newblock Slow transition to low-dimensional chaos in heavy-tailed recurrent neural networks. \newblock In {\em Advances in Neural Information Processing Systems (NeurIPS 2025)}, 2025. arXiv:2505.09816.

\bibitem{yin-bai-krishnaiah:88}
{Y.~Q.~Yin, Z.~D.~Bai and P.~R.~Krishnaiah. \newblock On the limit of the largest eigenvalue of the large dimensional sample covariance matrix. \newblock {\em Probab. Theory Related Fields}, 78(4):509--521, 1988.}

\end{thebibliography}
\end{document}